\documentclass[final,12pt]{colt2025} 
\usepackage{amsmath,amssymb}
\usepackage{graphicx,color}
\usepackage{booktabs}
\usepackage[normalem]{ulem}
\usepackage{mathtools}
\usepackage{multirow}
\usepackage{algorithm}
\usepackage[noend]{algpseudocode}
\usepackage{tikz} 
\usepackage{accents}
\usepackage{bm}
\usepackage{stmaryrd}
\usepackage{nicematrix}
\NiceMatrixOptions{code-for-first-row = \color{lightgray},
 code-for-first-col = \color{lightgray},
}

\usepackage{cancel}

\usepackage{soul} 
\usepackage{enumitem}
\usepackage{subcaption}

\newcommand{\cK}{\mathcal{K}}

\newcommand{\R}{\mathbb{R}}

\newcommand{\N}{\mathbb{N}}

\DeclareMathOperator*{\Tr}{Tr}

\newcommand*{\eps}{\varepsilon}

\DeclareMathOperator*{\argmin}{argmin}

\renewcommand{\leq}{\leqslant}
\renewcommand{\geq}{\geqslant}

\providecommand{\angs}[2]{\left\langle{#1},{#2}\right\rangle}
\providecommand{\norm}[1]{\lVert{#1}\rVert}

\newcommand*{\prox}{\text{prox}}

\newcommand{\db}[1]{\left\llbracket #1 \right\rrbracket}
\newcommand{\dbsmall}[1]{\llbracket #1 \rrbracket}

\newcommand{\blambda}{\bar{\lambda}}
\newcommand{\ulambda}{\underline{\lambda}}
\newcommand{\bmu}{\bar{\mu}}
\newcommand{\umu}{\underline{\mu}}
\newcommand{\rec}{\mathrm{rec}}
\newcommand{\sparse}{\mathrm{sp}}
\newcommand{\lr}{\mathrm{lr}}

\title[Accelerating Proximal Gradient Descent via Silver Stepsizes]{Accelerating Proximal Gradient Descent via Silver Stepsizes}
\usepackage{times}

\coltauthor{\Name{Jinho Bok} \Email{jinhobok@upenn.edu}\and
 \Name{Jason M. Altschuler} \Email{alts@upenn.edu}\\
 \addr University of Pennsylvania, Philadelphia, PA, USA}

\begin{document}

    \maketitle

\begin{abstract}
     	Surprisingly, recent work has shown that gradient descent can be accelerated without using momentum---just by judiciously choosing stepsizes. An open question raised by several papers is whether this phenomenon of stepsize-based acceleration holds more generally for constrained and/or composite convex optimization via projected and/or proximal versions of gradient descent. We answer this in the affirmative by proving that the silver stepsize schedule yields analogously accelerated rates in these settings. These rates are conjectured to be asymptotically optimal among all stepsize schedules, and match the silver convergence rate of vanilla gradient descent~\citep{ap23b, ap23a}, namely $O(\eps^{- \log_{\rho} 2})$ for smooth convex optimization and $O(\kappa^{\log_\rho 2} \log 1/\eps)$ under strong convexity, where $\eps$ is the precision, $\kappa$ is the condition number, and $\rho = 1 + \sqrt{2}$ is the silver ratio. The key technical insight is the combination of recursive gluing---the technique underlying all analyses of gradient descent accelerated with time-varying stepsizes---with a certain Laplacian-structured sum-of-squares certificate for the analysis of proximal point updates. 
\end{abstract}

\begin{keywords}
  Convex optimization, composite optimization, acceleration, proximal gradient descent, multi-step descent, silver stepsizes
\end{keywords}

\section{Introduction}

First, consider the fundamental setting of unconstrained smooth convex optimization $\min_{x} f(x)$, where without loss of generality $f$ is $1$-smooth (i.e., $\nabla f$ is $1$-Lipschitz). Until recently, folklore wisdom argued that the gradient descent algorithm (GD)
\begin{align*}
	x_{t+1} = x_t - \alpha_t \nabla f(x_t)
\end{align*}
requires $O(\eps^{-1})$ iterations to converge to an $\eps$-optimal solution, and moreover that any asymptotically faster convergence rate (i.e., ``accelerated'' rate) requires changing GD by adding momentum, internal dynamics, or other additional building blocks beyond just modifying the stepsizes.
See, for example, the textbooks~\citet{bubeckbook, beckbook, nesterovbook, accelsurvey} for a further discussion of mainstream approaches for accelerating GD.

\par Surprisingly, a recent line of work has shown that GD can be accelerated simply by using a judicious choice of stepsizes~\citep{altschulermsthesis,daccachemsthesis,eloimsthesis, gsw23, bnbpep, ap23b, ap23a, gdlongsteps, gswcomposing, gsw24, rppa, zhangcomposing, anytimeaccel}. In particular,~\citet{ap23b} showed that GD can converge at the \emph{silver convergence rate} $O(\eps^{-\log_{\rho} 2}) \approx O(\eps^{-0.7864})$---conjectured to be asymptotically optimal among all possible stepsize schedules---by using the silver stepsize schedule
\begin{align*}
	\alpha_t = \rho^{\nu(t+1)-1} + 1\,,
\end{align*}
where $\rho := 1 + \sqrt{2}$ denotes the silver ratio, and $\nu(i)$ denotes the $2$-adic valuation of $i$, i.e., the largest integer $j$ such that $2^j$ divides $i$. The silver stepsize schedule deviates qualitatively from mainstream stepsize schedules: it is time-varying, nonmonotone, fractal-like, and uses arbitrarily large stepsizes that are in particular larger than $2$ (the threshold at which constant stepsize schedules make GD divergent). See Figure~\ref{fig:sss} for an illustration. These results extend in a conceptually identical way if $f$ is additionally strongly convex: the silver stepsize schedule similarly accelerates the convergence of GD from the classical rate $O(\kappa \log \tfrac{1}{\eps})$ for constant stepsize schedules, to $O(\kappa^{\log_{\rho} 2} \log \frac{1}{\eps})$, where $\kappa$ denotes the condition number~\citep{ap23a}. 

\begin{figure}
    \begin{minipage}[c]{0.5\textwidth}
        \includegraphics[width=\textwidth]{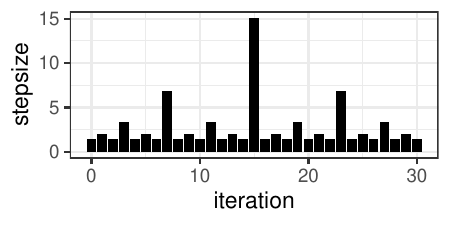}
    \end{minipage}
    \begin{minipage}[l]{0.5\textwidth}
        \caption{\footnotesize First $31$ stepsizes of the silver stepsize schedule. The stepsizes are time-varying and fractal-like.} \label{fig:sss}
    \end{minipage}
\end{figure}

\par In just the past year since, an exciting flurry of work has further investigated this phenomenon of stepsize-based acceleration:~\citet{gswcomposing, gsw24, rppa, zhangcomposing} improved the hidden constant in the aforementioned rate $O(\eps^{-\log_{\rho} 2})$ and provided extensions to different performance metrics, and~\citet{anytimeaccel} extended final-iterate convergence to anytime convergence, answering a COLT open question of~\citet{openproblem}. See the related work in \S\ref{subsec:priorwork} for a detailed discussion.

\paragraph*{Generality of stepsize-based acceleration?} However, these results focus primarily on the setting of unconstrained, smooth, (strongly) convex optimization. How general is this phenomenon? In particular,~\citet{gdlongsteps,ap23a} asked whether stepsize-based acceleration holds in the following two fundamental settings:

\begin{itemize}
	\item \textbf{Constrained optimization via projected GD?} For constrained optimization $\min_{x \in \cK} f(x)$, the standard analog of GD is \emph{projected GD} $x_{t+1} = \Pi_{\mathcal{K}}(x_t - \alpha_t \nabla f(x_t))$,
	where $\Pi_{\cK}$ denotes the projection operator onto the constraint set $\cK$, which is typically assumed to be convex and closed. Can projected GD be similarly accelerated by judiciously choosing stepsizes?
	\item \textbf{Composite optimization via proximal GD?} For composite convex optimization $\min_{x} F(x) := f(x) + h(x)$,
	where $f$ is convex and smooth, and $h$ is convex (but potentially non-smooth) and accessible through its proximal operator $\prox_{\alpha h}(x) = \argmin_{z} h(z) + \tfrac{1}{2\alpha}\|z-x\|^2$, the standard analog of GD is \emph{proximal GD} $x_{t+1} = \prox_{\alpha_t h}(x_t - \alpha_t \nabla 	f(x_t))$, also known as composite GD or backward-forward splitting.
	The point of proximal GD is that it enables better convergence than directly applying (sub)gradient descent on $F$. Indeed, since $F$ is non-smooth, GD on $F$ might not converge and even when it does, the convergence rate is slow; for example, under standard additional assumptions such as bounded gradients, it converges at the rate $O(\eps^{-2})$, whereas proximal GD enjoys a faster convergence rate $O(\eps^{-1})$ using constant stepsizes (see e.g., \citet{kelley, beckbook,  subgdrate}). Can this be accelerated further to the silver convergence rate using time-varying stepsizes?

\end{itemize}

Both of these settings are more general than the setting of previous results for stepsize-based acceleration. Indeed, projected GD recovers vanilla GD by taking $\cK = \R^d$, and proximal GD recovers vanilla GD by taking $h \equiv 0$. (Note also that proximal GD recovers projected GD by taking $h$ to be the indicator function of the constraint set.) Thus, we cannot hope to accelerate projected GD or proximal GD beyond what is possible for vanilla GD---but can we at least match it? Several recent works~\citep{latafat2024adaptive,ranjan2024verification,sambharya2024learning} have empirically observed that the silver stepsizes (or other time-varying schedules with large steps) can be effective for both settings;
however, theoretical guarantees are unknown beyond vanilla GD.

\paragraph*{Challenges.} At first glance, these extensions might seem straightforward since classical analyses for vanilla, projected, and proximal GD parallel each other closely and lead to nearly identical $O(\eps^{-1})$ convergence rates for constant stepsize schedules. For example, at least intuitively, the classical analysis of projected GD reduces to the analysis of vanilla GD by arguing that since the projection operator is $1$-Lipschitz, it cannot worsen the progress towards the minimizer; see e.g.,~\citet[\S3.1.2]{bansalgupta}. However, for stepsize-based acceleration it is unclear how to apply any such arguments. This is because existing analyses for stepsize-based acceleration use more complex dependencies between the iterates, in contrast to the modular argument of combining ``GD makes progress'' and ``interleaving projections does not worsen progress'' at each step; 
see the techniques section \S\ref{subsec:techniques} for a more detailed discussion.

\par This is not just a failure of current analyses. Interleaving projection and/or proximal operators can change the behavior of GD in fundamental ways when using time-varying stepsizes. 
For example,~\citet[\S4]{gdlongsteps} numerically investigated stepsize-based acceleration for $n=2,3$ steps and observed that projected GD has strictly worse rates than vanilla GD, and moreover proximal GD has strictly worse rates than projected GD. This suggests that as the level of generality increases, the algorithms may have to take strictly less aggressive stepsizes. As another example, even for the seemingly simple setting of quadratic objectives $f$---a setting for which stepsize-based acceleration has been known for $70+$ years~\citep{young}---convergence results do not extend in a straightforward way from vanilla GD to projected GD. 
This is due to fundamental issues beyond just a failure of existing analyses: for example, in the case $n=2$, it can be shown that the optimal stepsizes for unconstrained quadratic optimization (based on Chebyshev polynomials) provide a convergence rate greater than $1$ (i.e., not even contractive) if a projection or proximal operation is interleaved.

\begin{table}[]
	\scriptsize
	\renewcommand{\arraystretch}{1.2} 
	\begin{tabular}{|c|ccc|}
		\hline
		& \textbf{Convex opt}                      & \textbf{Constrained convex opt}                             & \textbf{Composite convex opt}                                                   \\ \hline
		Problem                     & $\min_x f(x)$                            & $\min_{x \in \mathcal{K}} f(x)$                                       & $\min_x f(x) + h(x)$                                                    \\
		Algorithm                   & GD                                       & Projected GD                                                & Proximal GD                                                                     \\
		Update rule                 & $x_{t+1} = x_t - \alpha_t \nabla f(x_t)$ & $x_{t+1} = \Pi_{\cK}(x_t - \alpha_t \nabla f(x_t))$ & $x_{t+1} = \mathrm{prox}_{\alpha_t h} \big( x_t - \alpha_t \nabla f(x_t))$ \\
		Rate (constant stepsizes) & $\eps^{-1}$ (folklore)    & $ \eps^{-1} $ (folklore)    & $\eps^{-1}$ (folklore)         \\
		Rate (silver stepsizes) & $\eps^{-\log_{\rho} 2}$  & \color{blue}{$\eps^{-\log_{\rho} 2}$ (Theorem~\ref{thm:main})}  & \color{blue}{$\eps^{-\log_{\rho} 2}$ (Theorem~\ref{thm:main})} \\
		\hline
	\end{tabular}
	\caption{\footnotesize 
		The mainstream convergence rate for GD and its variants is $O(\eps^{-1})$, achieved by constant stepsize schedules.~\citet{ap23b} recently showed that GD can be accelerated to $O(\eps^{-\log_\rho 2}) \approx O(\eps^{-0.7864})$ for unconstrained smooth convex optimization  (left) by using the silver stepsize schedule. 
		It remained open if this phenomenon of stepsize-based acceleration holds in the more general settings of constrained (middle) and composite (right) convex optimization. We show that the answer is yes.
		In this table, $\eps$ is the precision, $\rho = 1+\sqrt{2}$ is the silver ratio, $f$ and $h$ are convex, and $f$ is $1$-smooth (wlog by rescaling). If $f$ is also strongly convex, then the rates improve similarly from $O(\kappa \log 1/\eps)$ for constant stepsizes, to $O(\kappa^{\log_{\rho}2} \log 1/\eps)$ for the silver stepsize schedule, where $\kappa$ is the condition number; see Corollary~\ref{cor:sc}.
	}
	\label{tab:summary}
\end{table}

Hence, it remained unclear whether stepsize-based acceleration extends to more general settings of constrained and/or composite convex optimization via projected and/or proximal versions of GD.

\subsection{Contribution}

In this paper, we answer these questions in the affirmative. In particular, we show that by using the silver stepsize schedule, projected and proximal GD converge at the silver convergence rate. This was previously only known for vanilla GD. See Table~\ref{tab:summary} for a summary.

We state our main result for proximal GD (which generalizes projected GD) and for general $M$-smoothness (equivalent after normalizing the silver stepsizes by $1/M$). For context, note that the tight rate for proximal GD using constant stepsizes $\alpha_t \equiv 1$ is $F(x_n) - F(x_*) \leq \frac{1}{4n}  M\norm{x_0 - x_*}^2$ where $x_*$ denotes any minimizer of $F$ \citep[Theorem 7]{tv23}.

\begin{theorem}[Main result]\label{thm:main} 
	For any horizon $n = 2^{k} - 1 (k \in \N)$, any $M$-smooth convex function $f$, any convex function $h$, and any initialization $x_0$, 
	\[
		F(x_n) - F(x_*) \leq \frac{\rho}{4\sqrt{2} n^{\log_2 \rho}}  M\norm{x_0 - x_*}^2 \,,
	\]
	where $x_*$ denotes any minimizer of $F = f + h$, and $x_n$ denotes the $n$-th iterate of proximal GD using the silver stepsize schedule. In particular, in order to achieve error $F(x_n) - F(x_*) \leq \eps$, it suffices to run proximal GD for some number of iterations $n$, where 
	\begin{align*}
		n \lesssim \left( M\|x_0 - x_*\|^2 / \eps \right)^{\log_{\rho} 2} \approx \left( M\|x_0 - x_*\|^2 / \eps \right)^{0.7864}\,.
	\end{align*}
\end{theorem}

We make four remarks. First, by standard black-box reductions (see, e.g.,~\citet[Footnote 2]{ap23b}), this acceleration immediately extends to the strongly convex setting. This improves the tight iteration complexity $O(\kappa \log 1/\eps)$ for proximal GD when using constant stepsizes \citep{proxgdsc}, to match the silver convergence rate $O(\kappa^{\log_{\rho} 2} \log 1/\eps)$ that was previously only known for vanilla GD.

\begin{corollary}\label{cor:sc}
	Assume further that $f$ is $m$-strongly convex. There exists a stepsize schedule such that $\norm{x_n - x_*} \leq \eps$ after
	\begin{align*}
		n \lesssim \kappa^{\log_\rho 2}  \log \left(  \|x_0 - x_*\| / \eps \right) \approx \kappa^{0.7864}  \log \left( \|x_0 - x_*\| / \eps \right) 
	\end{align*}
	iterations, where $\kappa := M/m$ denotes the condition number.
\end{corollary}

Second, these acceleration results readily imply the silver convergence rate for all special cases of proximal GD, e.g.: vanilla GD ($h \equiv 0$),  projected GD ($h = \iota_{\mathcal{K}})$, and  ISTA~\citep{ista} ($h = \norm{\cdot}_1$). For the latter two cases, no results on stepsize-based acceleration were known.

Third, for simplicity, these results are stated for horizons $n = 2^{k} - 1$. However, the asymptotic rates extend to arbitrary $n$ by standard tricks, e.g., output the best iterate rather than the final iterate, or run for the largest $n' = 2^{k'} - 1$ less than $n$. Finally, we remark about the tightness of our bounds.

\begin{remark}[Optimality of the bounds] 
	In our rates $O(\eps^{-\log_{\rho} 2})$ and $O(\kappa^{\log_{\rho} 2} \log 1/\eps)$, the exponent $\log_{\rho} 2$ is tight for the silver stepsize schedule (proven in Appendix \ref{appendix:lowerbound}), and is conjecturally optimal among all possible stepsize schedules (since that is conjectured even for the special case of GD~\citep{ap23b,  ap23a, gswcomposing, gsw24}). 
    
    These rates for stepsize-based acceleration are weaker than traditional momentum-based acceleration, which changes (vanilla, projected, proximal) GD beyond modifying stepsizes and achieves rates $O(\eps^{-1/2})$ and $O(\kappa^{1/2} \log 1/\eps)$ for the convex and strongly convex settings, respectively, both of which are information-theoretically optimal among first-order methods~\citep{accelsurvey}. 

\end{remark}

\subsection{Overview of techniques}\label{subsec:techniques}

The overarching rationale for stepsize-based acceleration is that although short and long steps are individually suboptimal---since short steps undershoot on flat objectives, and long steps overshoot on steep objectives---judiciously combining them can lead to overall faster convergence because these worst-cases cannot align. See~\citet{altschulermsthesis} for an interpretation in terms of hedging.
Notably, proving stepsize-based acceleration precludes using classical analyses (such as those for constant stepsize schedules, see e.g.,~\citet{bansalgupta}) which bound the progress of each iteration separately and then sums these bounds. Indeed, any such analysis is unable to prove any benefit of deviating from constant stepsize schedules.  

\par Instead, in order to prove stepsize-based acceleration, it is necessary to argue holistically about the cumulative progress---i.e., the \emph{multi-step descent}---of the algorithm. 
Let us briefly describe how prior work accomplishes this for GD, before moving to the present setting of proximal GD.

\paragraph*{Stepsize acceleration for GD.}  
Suppose one seeks to prove that $n$ steps of GD converges at a rate $\tau_n$, i.e., an inequality 
$f(x_n) - f(x_*) \leq \tau_n \|x_0 - x_*\|^2$.
All existing analyses of stepsize-based acceleration establish this by proving an identity of the form
\begin{align}
	  f(x_*) - f(x_n)  + \tau_n \norm{x_0 - x_*}^2 = \sum_{i, j} \lambda_{ij}Q_{ij} + (\mathrm{SOS})\,,
	\label{eq:intro-descent}
\end{align}
where $(\mathrm{SOS})$ is a sum-of-squares polynomial, $\lambda_{ij} \geq 0$, and $Q_{ij} \geq 0$ are co-coercivities (i.e., certain polynomials in the iterates $x_i$ and $x_j$, gradients $\nabla f(x_i)$ and $\nabla f(x_j)$, and function values $f(x_i)$ and $f(x_j)$; see Definition~\ref{def:coco}). The identity~\eqref{eq:intro-descent} constitutes a \emph{certificate} of the desired multi-step descent since the right hand side is obviously nonnegative. 

\par In words, the co-coercivity inequality $Q_{ij} \geq 0$ is a long-range consistency condition regarding information about $f$ at different iterates along the GD trajectory. 
Crucially, these inequalities are valid not only for consecutive iterates but for \emph{all} $x_i$ and $x_j$, enabling a holistic argument that captures how iterations affect each other. In the right hand side of \eqref{eq:intro-descent}, $\lambda_{ij} \geq 0$ should be interpreted as multipliers that combine the inequalities $Q_{ij} \geq 0$, and the SOS term should be interpreted as a way to add any other possible valid inequalities. Indeed, although we will not use the following fact in this paper, identities of the form~\eqref{eq:intro-descent} provide a complete proof system in the sense that any valid convergence rate $\tau_n$ admits a proof via such an identity~\citep{thg17}. 

\par There are two challenges for proving an identity of the form~\eqref{eq:intro-descent}: 1) providing explicit expressions for the multipliers $\{\lambda_{ij}\}$ and the SOS term, and 2) verifying the identity, which amounts to checking that the coefficients match on both sides of this polynomial identity. In particular, the issue is that the complexity of both tasks ostensibly increases in the number of iterations~$n$.

\par The key technique that makes both steps tractable is \emph{recursive gluing}. This technique was introduced in~\citet{ap23b, ap23a} and is at the heart of all analyses in subsequent works on accelerating GD~\citep{gswcomposing, gsw24, rppa, zhangcomposing}. The idea is that for increasing $n$, the complexity of proving identities such as~\eqref{eq:intro-descent} can be controlled by using the recursive structures of $\{\lambda_{ij}\}$ and $(\mathrm{SOS})$ that are inherited from the recursive pattern of the stepsizes. This allows analyzing the performance of a longer schedule created by concatenating two shorter schedules---possibly with additional modifications such as changing~\citep{ap23a} or adding~\citep{ap23b,gswcomposing,gsw24, rppa, zhangcomposing} stepsizes---by essentially composing the progress of the shorter schedules. More precisely, this provides a new identity certifying multi-step descent for the longer schedule by essentially combining the identities for the shorter schedules. The multipliers $\{\lambda_{ij}\}$ for the new identity are comprised of three components in a formulaic manner: a recursive component (that combines the identities for the shorter subsequences), a sparse correction (that affects only $O(1)$ entries of $\{\lambda_{ij}\}$), and a low-rank correction (that affects only $O(1)$ rows of $\{\lambda_{ij}\}$). The upshot is that analysis for the longer sequence is inherited from analyses for the shorter sequences.
Concretely, this reduces proving an $n$-step identity~\eqref{eq:intro-descent} to verifying only $O(1)$ coefficients.

\paragraph*{Stepsize acceleration for proximal GD.} We analyze proximal GD by proving a multi-step descent identity analogous to~\eqref{eq:intro-descent}, but now with a composite objective $F = f+h$, with co-cocoercivity inequalities $Q_{ij}^{f} \geq 0$ and $Q_{ij}^{h} \geq 0$ which respectively encode consistency conditions for $f$ and $h$, and with two sets of corresponding multipliers $\{\lambda_{ij}\}$ and $\{\mu_{ij}\}$. The identity is now of the form
\begin{align}
    F(x_*) - F(x_n) 
    + \tau_n \norm{x_0 - x_*}^2 = \sum_{i, j} \lambda_{ij}Q_{ij}^{f} + \sum_{i,j} \mu_{ij} Q_{ij}^{h} + (\mathrm{SOS})\,.
	\label{eq:intro-descent-proximal}
\end{align}
The immediate challenge in our setting is to find the multipliers and the SOS term. Our starting point is the observation that we can reuse the same multipliers $\{\lambda_{ij}\}$ from the analysis of GD. While this approach is seemingly straightforward, the corresponding analysis is rather subtle due to the dependency between the co-coercivities. Based on this, we build the multipliers $\{\mu_{ij}\}$ (which are new and different from $\{\lambda_{ij}\}$) that can also be characterized using the technique of recursive gluing. Combined with $\{\lambda_{ij}\}$, this reduces the complexity of proving the $n$-step identity into verifying $O(1)$ coefficients as described above.

\par Conceptually, the SOS term poses a core new difficulty that arises for proximal GD but not GD. Given $\{\lambda_{ij}\}$ and $\{\mu_{ij}\}$, this term is the residual in the identity~\eqref{eq:intro-descent-proximal}; the challenge is proving that the resulting quadratic form is actually a sum of squares.  For GD, this is easy as the term is just a single square~\citep{ap23b}.\footnote{In fact, in the strongly convex case, the SOS term can be made to be zero for GD~\citep{ap23a}.} In contrast, for proximal GD this term is of \emph{high rank}, implying that any decomposition as a sum of squares (if one exists) requires $\Omega(n)$ terms. Certifying that this term is SOS by finding such a decomposition therefore appears to bring us back to the original challenge of proving multi-step descent identities: the complexity increases in $n$. 

\par We approach this by combining certain aspects of the analysis of both components of proximal GD: vanilla GD and proximal point method. This is most easily stated in terms of the coefficient matrix corresponding to this quadratic form; establishing that the term is SOS is equivalent to showing that this matrix is positive semidefinite. In the case of GD, this matrix has rank $1$; in the case of the proximal point method \citep{taylorcomposite}, while the matrix has rank $\Omega(n)$, it has an elegant Laplacian structure which implies positive semidefiniteness. It turns out that modulo a Schur complement, the matrix in our analysis of proximal GD is precisely a sum of a rank-1 matrix and a Laplacian matrix that can be constructed recursively. This provides a simple certificate of positive semidefiniteness and thus provides the final missing piece---verifying that the term is indeed SOS---for the recursive gluing argument to prove the desired identity~\eqref{eq:intro-descent-proximal}. We remark that such combination of structures is novel to analyses of stepsize acceleration, 
and we are hopeful that such ideas might lead to new possibilities in this line of work.

\subsection{Related work}\label{subsec:priorwork}

\paragraph*{Stepsize-based acceleration.}
For the special case of quadratic objectives, Young showed in 1953 that GD can be accelerated using non-constant stepsizes, given explicitly as the inverses of the roots of Chebyshev polynomials~\citep{young}. However, until recently it was believed that this phenomenon of stepsize-based acceleration was limited to quadratic objectives. Indeed, despite the investigation of many stepsize schedules---including even adaptive strategies such as exact line search~\citep{boyd,nocedal,de2017worst}, Armijo-Goldstein schedules~\citep{armijo}, Polyak-type schedules~\citep{polyakbook}, and Barzilai-Borwein-type schedules~\citep{barzilai1988two}---there had been no analysis showing that any stepsize schedule improves over constant stepsizes in any setting beyond quadratics. 

Beginning with~\citet{altschulermsthesis}, a line of work showed that such results are in fact possible. Most relevant to the setting of this paper,~\citet[Chapter 8]{altschulermsthesis} provides optimal stepsizes for $n=2,3$ iterations for smooth, strongly convex optimization. Cycling through these schedules provides a constant factor improvement over the unaccelerated rate for constant stepsizes. These results were proved via multi-step descent identities of the form~\eqref{eq:intro-descent} by leveraging the performance estimation problem (PEP) framework of~\citet{dt14, thg17}, which uses semidefinite programming to numerically compute the convergence rate of a \emph{given} stepsize schedule. A key difficulty for PEP is that the \emph{search} for optimal stepsizes is nonconvex and computationally difficult as $n$ increases.~\citet{daccachemsthesis,eloimsthesis} similarly computed optimal stepsizes for $n=2,3$ iterations for the non-strongly convex setting and different performance metrics.~\citet{bnbpep} developed a branch-and-bound framework for PEP in order to numerically perform this nonconvex search for stepsizes, and computed schedules up to $n=50$ for the non-strongly convex setting.~\citet{gdlongsteps} obtained a constant-factor improvement over constant stepsizes for the non-strongly convex setting by cycling through approximate schedules of length $n=127$.

\par In September 2023, two concurrent lines of work showed asymptotic acceleration for GD.~\citet{gsw23} proposed a nonperiodic stepsize schedule similar to but different than the silver schedule, and showed that it accelerates GD at the rate $O(\eps^{-0.9467})$ for smooth convex optimization, and $O(\kappa^{0.9467} \log 1/\eps)$ for the strongly convex setting. Concurrently,~\citet{ap23b, ap23a} proved that GD can accelerate at the silver rate $O(\eps^{-\log_{\rho} 2}) \approx O(\eps^{-0.7864})$, by introducing the silver stepsizes and developing the recursive gluing technique to analyze it. The results are analogous for the strongly convex setting, yielding the rate $O(\kappa^{\log_{\rho} 2} \log \tfrac{1}{\eps}) \approx O(\kappa^{0.7864} \log \tfrac{1}{\eps})$. These rates are conjecturally asymptotically optimal among all possible stepsize schedules. 

\par In the year since, a flurry of work has built upon these results.
One thrust refined the hidden constant in the asymptotic rate $O(\eps^{-\log_{\rho} 2})$ of~\citet{ap23b}, by a factor of $2$ for the silver stepsizes~\citep{rppa}, by a factor of ${\sim}2.32$ using a similar schedule~\citep{gsw24}, and most recently by a conjecturally optimal factor of ${\sim}2.37$ using highly optimized schedules~\citep{gswcomposing,zhangcomposing}. A second thrust of these papers
extended convergence results from function suboptimality to performance metrics related to the gradient norm. For these results,~\citet{gswcomposing, zhangcomposing} identified certain families of stepsize schedules that allow structured formulae for concatenation (and adding a stepsize in the middle), in this way refining and simplifying the recursive gluing technique. Leveraging this,~\citet{anytimeaccel} answered a COLT open question of~\citet{openproblem} by concatenating silver stepsize schedules of varying lengths, in order to establish ``anytime convergence'' results which ensure accelerated convergence (at rate $O(\eps^{-0.8932})$) for GD at all iterations rather than just at the final iteration.
It seems plausible that similar results extend to proximal GD. 

All these papers\footnote{
    The one exception is~\citet{rppa}, which shows that the relaxed proximal point algorithm (RPPA) achieves the silver convergence rate for unconstrained convex (but possibly nonsmooth) $f$. However, this does not imply our results as it requires access to a stronger oracle: the proximal operator for the entire objective. Indeed, as detailed in~\citet[\S1.1]{rppa}, RPPA on $f$ is equivalent to GD on the (smooth and convex) Moreau envelope of $f$, which means that its setting is in a formal sense equivalent to unconstrained smooth convex optimization---the setting of previous stepsize-based acceleration results for GD. See the discussion around~\citet[Theorem 5.3]{rppa}. 
} focus on unconstrained minimization of smooth and (strongly) convex objectives; the purpose of the present paper is to go beyond this setting.

\paragraph*{First-order methods for composite convex optimization.}

Minimizing composite objectives is a fundamental problem with many applications; see, e.g., the survey~\citet{proxsurvey}. A prominent example is when $h = \norm{\cdot}_1$, where proximal GD is known as ISTA \citep{ista}. The seminal work~\citet{fista} proposed the FISTA algorithm and showed an $O(\eps^{-1/2})$ accelerated rate. Since then, many algorithms have been proposed; see, e.g.,~\citet[Chapter 10]{beckbook}. 
A recent line of work has used PEP not only to obtain tight analyses of existing algorithms \citep{taylorcomposite,proxgdsc,accelsurvey}, but also to develop an optimal algorithm with rate $O(\eps^{-1/2})$ that has the best possible constant factor and an exactly matching lower bound~\citep{optista}. 
As is the case for GD, such acceleration results in the composite setting are obtained by changing proximal GD beyond just modifying stepsizes, e.g., via momentum.

\paragraph*{Time-varying and/or large stepsizes in other contexts.} Time-varying and/or large stepsizes have also been investigated in other contexts, such as accelerating GD when the objective has Hessians with bimodal~\citep{oymak2021provable} or multiscale~\citep{kelner2022big} spectral structure, when the objective is separable~\citep{altschulermsthesis, ap24random}, when the objective has fourth order growth~\citep{davis2024gradient}, as well as for logistic regression~\citep{axiotis2023gradient,wu2024large}, parametric convex optimization~\citep{ranjan2024verification,sambharya2024learning}, min-max optimization~\citep{shugart2025negative}, and training neural networks~\citep{smith2017cyclical,huang2017snapshot}. Large steps for (stochastic) GD can lead to phenomena such as the Edge of Stability~\citep{wu2018eos, cohen2021gradient}, implicitly biased convergence towards flat minima~\citep{keskar2017large},
stochastic-like behavior on multiscale functions~\citep{kong2020stochasticity}, and chaos or catapults along the trajectory~\citep{van2012chaotic, lewkowycz2020large, wang2023good, chen2024stability}. 

\par Interestingly, recent empirical results document that the silver stepsize schedule performs well in constrained settings~\citep{ranjan2024verification}; also, across several benchmarks, hyperparameter tuning algorithms empirically learn time-varying stepsizes schedules for projected and proximal GD that qualitatively resemble the silver stepsize schedule in certain aspects such as the use of large but rare stepsizes, and fractal-like or near-periodic structure~\citep{latafat2024adaptive,sambharya2024learning}. This paper provides theoretical justification for these empirical phenomena.

\section{Preliminaries and notations}

\paragraph*{Notation.} We denote the sum of a finite sequence $a = \{a_i\}$ by $\sum a$. For vectors $a, b$ of the same length, $a \geq b$ denotes $a_i \geq b_i$ for all $i$. Vectors are assumed to be vertical. We use $e_1, e_2, \dots$ to denote the standard basis vectors; the ambient dimension will always be clear from context. For simplicity, throughout $M=1$; this is without loss of generality after rescaling the functions $f$ and $h$ by $1/M$. Henceforth, with a slight abuse of notation, $\{\alpha_i\}$ denotes the (normalized) silver stepsizes.

\subsection{Co-coercivities}

For shorthand, we denote $f_i := f(x_i)$, $g_i := \nabla f(x_i)$, $h_i := h(x_i)$, $F_i := F(x_i)$ and $s_i$ to be a subgradient of $h$ at $x_i$, for all $i \in \{0, 1, \dots, n, *\}$. In this notation, proximal GD has the update
$x_{t+1} = x_t - \alpha_t(g_t + s_{t+1})$,
and the optimality condition is $g_* + s_* = 0$. As described in \S\ref{subsec:techniques}, our analysis uses certain inequalities---\emph{co-coercivities}---between pairs of iterates.

\begin{definition}[Co-coercivities]\label{def:coco} Define $ Q_{ij}^f := f_i - f_j - \angs{g_j}{x_i - x_j} - \frac{1}{2}\norm{g_i - g_j}^2$ and $Q_{ij}^h := h_i - h_j - \angs{s_j}{x_i - x_j}$.
\end{definition}
\begin{lemma}[{\citet[Theorem 4]{thg17}}]\label{lem:coco} Let $f$ be convex and 1-smooth, and let $h$ be convex. Then $Q_{ij}^f \geq 0$ and $Q_{ij}^h \geq 0$ for all $i, j$.   
\end{lemma}

\subsection{Helper lemmas}\label{ssec:helper}

For each $k \in \N$ and $n = 2^k-1$, we denote by $\pi^{(k)} := [\alpha_0, \dots, \alpha_{n-1}]$ the first $n$ entries of the silver stepsize schedule. We make use of its following basic properties~\citep{ap23b}.

\begin{lemma}[Properties of $\pi^{(k)}$]\label{lem:silver} For all $k \in \N$, $\sum \pi^{(k)} = \rho^k-1$ and  $\pi^{(k+1)} = [\pi^{(k)}, \rho^{k-1} + 1, \pi^{(k)}]$.
\end{lemma}

Our analysis uses an auxiliary sequence 
$c^{(k)}$ which is related to but different from $\pi^{(k)}$.

\begin{definition}\label{def:ck} For all $k \in \N$ and $n = 2^k-1$, the sequence $c^{(k)} \in \R^n$ is recursively defined as $c^{(1)} := [2(\rho-1)], c^{(k+1)} := [\pi^{(k)}, (1+1/\rho^k)(\rho^{k-1}+1), \rho c^{(k)} - (\rho-1-1/\rho^k)\pi^{(k)}]$.
\end{definition}

\begin{lemma}[Properties of $c^{(k)}$]\label{lem:ck} For all $k \in \N$, $\sum c^{(k)} = 2(\rho^k-1)$ and $c^{(k)} \geq \pi^{(k)}$.
\end{lemma}
\begin{proof} For the equality, by induction, $\sum c^{(k+1)} = (\rho^k-1) + (1+1/\rho^k)(\rho^{k-1}+1) + 2\rho(\rho^k-1) - (\rho-1-1/\rho^k)(\rho^k-1) = 2(\rho^{k+1} - 1)$ from $\sum \pi^{(k)} = \rho^k - 1$ (Lemma \ref{lem:silver}). For the inequality, this is straightforward by induction and the recursive construction of $\pi^{(k)}$ (Lemma \ref{lem:silver}).
\end{proof}

\section{Proof of the main result}\label{sec:proof}

In this section, we prove the main result (Theorem \ref{thm:main}) by constructing nonnegative multipliers and sum-of-squares terms such that the following multi-step descent identity holds. These terms capture how different iterations affect other iteration's progress, and in particular capture effects beyond consecutive iterations. See \S\ref{subsec:techniques} for a high-level overview of our analysis techniques. 

\begin{theorem}[Certificate of multi-step descent]\label{thm:maincert} 
Let $k \in \N$ and  $n = 2^k - 1$. There exist $\lambda = \lambda^{(k)}$ (Definition~\ref{def:lambda}), $\mu = \mu^{(k)}$ (Definition~\ref{def:mu}), and $S = S^{(k)}$ (Definition~\ref{def:laplacian}) such that:
\begin{itemize}
    \item [(i)] \underline{Nonnegativity of multipliers.} $\lambda_{i, j}, \mu_{i, j} \geq 0$ for all $i \neq j$.
    \item [(ii)] \underline{Positive semidefiniteness of slack.} $S \succeq 0$. 
    \item [(iii)] \underline{Multi-step descent identity.}
\end{itemize}
\begin{equation}\label{eq:maincert}
    \sum_{i, j} \lambda_{i,j} Q_{ij}^f + \sum_{i, j} \mu_{i,j} Q_{ij}^h = (2\rho^k - 1) (F_* - F_n) + \frac{\rho}{2\sqrt{2}}\norm{x_0 - x_*}^2 - \frac{1}{2}(\norm{u}^2 + \Tr(VSV^T))\,,
\end{equation}
where for shorthand $u := x_0 - x_* - \sum_{i=0}^{n-1} \alpha_i g_i - \rho^k g_n - \sum_{j=1}^n c_j s_j - s_*$ with $c := c^{(k)}$ (Definition~\ref{def:ck}), and 
$V := \begin{bmatrix}
    x_0 - x_* \ | \  s_1\  |\ \dots \ | \ s_n \ | \ s_* 
\end{bmatrix}$ is the columnwise-concatenated matrix.

\end{theorem}

Theorem \ref{thm:maincert} immediately implies the main result (Theorem \ref{thm:main}).

\begin{proof}[Proof of Theorem \ref{thm:main}] The left hand side of the multi-step descent identity \eqref{eq:maincert} is nonnegative since for all $i \neq j$, the co-coercivities $Q_{ij}^f, Q_{ij}^h$ are nonnegative (Lemma \ref{lem:coco}) and so are the multipliers $\lambda_{i, j}, \mu_{i, j}$ (by item (i)). The term $\|u\|^2 + \Tr(VSV^T)$ on the right hand side is also nonnegative as a sum of squares, since $S= S^{(k)}$ is positive semidefinite (by item (ii)). Therefore, by rearranging,~\eqref{eq:maincert} implies the desired bound $
    F_n - F_* \leq \frac{\rho}{\sqrt{2}(4\rho^k-2)} \norm{x_0 - x_*}^2\leq \frac{\rho}{4\sqrt{2}n^{\log_2 \rho}} \norm{x_0 - x_*}^2$.
\end{proof}

The rest of the section is devoted to establishing Theorem \ref{thm:maincert}. In \S\ref{ssec:multipliers} we define the multipliers $\lambda, \mu$ and prove their nonnegativity (item (i)). In \S\ref{ssec:sos} we define the coefficient matrix $S$ for the SOS term and prove that $S$ is positive semidefinite (item (ii)). In \S\ref{ssec:verification} we prove the multi-step descent identity (item (iii)). The definitions are recursive in $k$ and the proofs follow by induction on $k$; the base case $k=1$ for Theorem~\ref{thm:maincert} amounts to an elementary identity and is deferred to Appendix \ref{appendix:basecase}.

\subsection{Construction of multipliers via recursive gluing}\label{ssec:multipliers}

First, we define the multipliers for the co-coercivities $Q_{ij}^f$ associated with $f$. Although we are considering the more general setting of proximal GD for composite optimization, it turns out that we can use the same multipliers from previous works on GD for single-objective convex optimization. Common to all such analyses, $\lambda^{(k+1)}$ is constructed via recursive gluing as a sum of three components: a recursion component which glues two copies from $\lambda^{(k)}$, a sparse correction which has $O(1)$ nonzero entries, and a low-rank correction which affects only $O(1)$ rows. The original recursive gluing construction of~\citet{ap23b} was refined recently in~\citet{rppa, gsw24}, and we use this version here; in particular, the sparse correction has exactly two nonzero entries, and the low-rank correction contains two rows that are multiples of $\pi^{(k)}$. For bookkeeping purposes, it is convenient to isolate the bottom row $\{\lambda_{*,j}\}$ of $\lambda$ into a separate vector $\ulambda$. See Figure~\ref{fig:cert}, left, in Appendix \ref{appendix:figure} for an illustration.

\begin{definition}[Multipliers for $f$]\label{def:lambda} For $k \in \N$ and $n = 2^k-1$, define $\lambda_{i, j}^{(k)}, i \in \{0, 1, \dots, n, *\}, j \in \{0, 1, \dots, n\}$ as
\[\lambda_{i, j}^{(k)} := \blambda_{i, j}^{(k)}\bm{1}_{\{i \neq *\}} + \ulambda_j^{(k)}\bm{1}_{\{i = *\}}\,,
\]
where $\ulambda^{(k)} := [\pi^{(k)}, \rho^k]$ and $\blambda^{(1)} := \ \begin{bNiceArray}{cc}[first-row, first-col]
        & 0 & 1 \\
        0 & 0 & \rho \\
        1 & 1 & 0 \\
    \end{bNiceArray}$,
\begin{align*}
    \blambda^{(k+1)} &:= \underbrace{\blambda^{(k+1), \rec}}_{\text{recursion}} \;\;\;+ \underbrace{\blambda^{(k+1), \sparse}}_{\text{sparse correction}} + \underbrace{\blambda^{(k+1), \lr}}_{\text{low-rank correction}}\,.
\end{align*}
The multipliers $\blambda^{(k+1), \rec}, \blambda^{(k+1), \sparse}, \blambda^{(k+1), \lr}$ are defined as
\begin{align*}
    \blambda^{(k+1), \rec}_{i, j} &:= \blambda^{(k)}_{i, j} \bm{1}_{\{0 \leq i, j \leq n\}} + \rho^2 \blambda^{(k)}_{i-n-1, j-n-1}\bm{1}_{\{n+1 \leq i, j \leq 2n + 1\}}\,, \\
    \blambda^{(k+1), \sparse}_{i, j} &:= \rho \bm{1}_{\{(i, j) = (n, 2n+1)\}} + \rho^k \bm{1}_{\{(i, j) = (2n+1, n)\}}\,, \\
    \blambda^{(k+1), \lr}_{i, j} &:= \rho \pi^{(k)}_{j-n} \bm{1}_{\{i = n, n+1 \leq j \leq 2n\}} + \rho \pi^{(k)}_{j-n} \bm{1}_{\{i = 2n+1, n+1 \leq j \leq 2n\}}\,.
\end{align*}
\end{definition}

Next, we recursively define the multipliers for the co-coercivities $Q_{ij}^{h}$ associated with $h$, which are new to our analysis for \emph{proximal} GD. Similarly to $\lambda^{(k+1)}$, this construction of $\mu^{(k+1)}$ is based on recursive gluing: we use a recursion component that glues two copies from $\mu^{(k)}$, a sparse correction component which has three nonzero entries, and a low-rank correction component which is comprised of linear combinations of $\pi^{(k)}$ and $c^{(k)}$. As for $\lambda$, for bookkeeping purposes we isolate the bottom row $\{\mu_{*,j}\}$ into a separate vector $\umu$. See Figure~\ref{fig:cert}, right, in Appendix \ref{appendix:figure} for an illustration.

\begin{definition}[Multipliers for $h$]\label{def:mu} For $k \in \N$ and $n = 2^k-1$, define $\mu_{i, j}^{(k)}, i \in \{1, \dots, n, *\}, j \in \{1, \dots, n\}$ as
\[\mu^{(k)}_{i, j} := \bmu^{(k)}_{i, j} \bm{1}_{\{i \neq *\}} + \umu^{(k)}_j \bm{1}_{\{i = *\}}\,,
\]
where $\umu^{(k)} := c^{(k)} + e_1$ and $\bmu^{(1)} := \ \begin{bNiceArray}{cc}[first-row, first-col]
        & 1 \\
        1 & 0 \\
    \end{bNiceArray}$,
\begin{align*}
    \bmu^{(k+1)} &:= \underbrace{\bmu^{(k+1), \rec}}_{\text{recursion}} \;\;\;+ \underbrace{\bmu^{(k+1), \sparse}}_{\text{sparse correction}} + \underbrace{\bmu^{(k+1), \lr}}_{\text{low-rank correction}}\,.
\end{align*}
The multipliers $\bmu^{(k+1), \rec}, \bmu^{(k+1), \sparse}, \bmu^{(k+1), \lr}$ are defined as
\begin{align*}
    \bmu^{(k+1), \rec}_{i, j} &:= \bmu^{(k)}_{i, j}\bm{1}_{\{1 \leq i, j \leq n\}} + \rho^2\bmu^{(k)}_{i-n-1, j-n-1}\bm{1}_{\{n+2 \leq i, j \leq 2n+1\}}\,, \\
    \bmu^{(k+1), \sparse}_{i, j} &:= \rho^k \bm{1}_{\{(i, j) = (n, n+1)\}} + \rho^2 \bm{1}_{\{(i, j) = (n+1, n+2)\}} + (\rho-\frac{1}{\rho^k})(\rho^{k-1}+1)\bm{1}_{\{(i, j) = (2n+1, n+1)\}}\,, \\
    \bmu^{(k+1), \lr}_{i, j} &:= (1-\frac{\rho^k}{\rho^{k-1}+1})(c^{(k)}_j - \pi^{(k)}_j)\bm{1}_{\{i = n, 1 \leq j \leq n\}} + \frac{\rho^k}{\rho^{k-1}+1}(c^{(k)}_j - \pi^{(k)}_j)\bm{1}_{\{i = n+1, 1 \leq j \leq n\}} \\
    &\quad+ \frac{\rho^k}{\rho^{k-1}+1}\pi^{(k)}_{j-n-1}\bm{1}_{\{i = n, n+2 \leq j \leq 2n+1\}} + \frac{\rho}{\rho^{k-1}+1}\pi^{(k)}_{j-n-1}\bm{1}_{\{i = n+1, n+2 \leq j \leq 2n+1\}} \\
    &\quad+ ((\rho+1)c^{(k)}_{j-n-1} - (1+\frac{1}{\rho^k})\pi^{(k)}_{j-n-1})\bm{1}_{\{i = 2n+1, n+2 \leq j \leq 2n+1\}}\,.
\end{align*}
\end{definition}

\begin{lemma}[Nonnegativity of multipliers]\label{lem:multpositive} For all $k \in \N$ and $i \neq j$,
    $\lambda_{i,j}^{(k)} \geq 0$ and $\mu_{i,j}^{(k)} \geq 0$.
\end{lemma}
\begin{proof}
 This follows by induction on $k$. The base case $k=1$ is trivial. For the inductive step, assuming $\lambda^{(k)}_{i, j} \geq 0$, clearly $\lambda^{(k+1)}_{i, j} \geq 0$ holds as the entries are constructed by adding nonnegative numbers. This is similarly true for nearly all entries of $\mu^{(k+1)}$ from the inequality $c^{(k)} \geq \pi^{(k)}$ (Lemma \ref{lem:ck}). It only remains to check the $n-1$ entries $\{\mu_{n,j}^{(k+1)}\}_{1 \leq j < n}$ where a nonpositive vector is added due to $\bmu^{(k+1), \lr}$; a quick calculation in Appendix~\ref{app:nonnegativity} shows that they are nonnegative.
\end{proof}

\subsection{Laplacian sum-of-squares}\label{ssec:sos}

Here we provide details on the slack matrix $S^{(k)}$ in the multi-step descent identity~\eqref{eq:maincert}. In particular we show that it is positive semidefinite, proving item (ii) of Theorem~\ref{thm:maincert}. Key to this result is the identification of a Laplacian matrix from its recursive structure.\footnote{Recall that a Laplacian matrix is a symmetric matrix such that all nondiagonal entries are nonpositive, and all row sums (equivalently, all column sums) are $0$. It is a classical fact that Laplacian matrices are positive semidefinite; this follows, e.g., from the Gershgorin circle theorem or from the observation $x^TLx = \sum_{i < j} (-L_{ij})(x_i - x_j)^2 \geq 0$. We refer to standard textbooks such as~\citet{chung1997spectral} for further background about Laplacian matrices.} We begin by defining $S^{(k)}$. Similarly to the constructions in \S\ref{ssec:multipliers}, this construction is recursive and involves $\pi^{(k)}$ and $c^{(k)}$.

\begin{definition}[Construction of $L^{(k)}$ and $S^{(k)}$]\label{def:laplacian} For $k \in \N$ and $n = 2^k-1$, define
\[L^{(k)} := \begin{bmatrix}
    \bar{L}^{(k)} & -(c^{(k)})^T \\
    -c^{(k)} & 2(\rho^k - 1)
\end{bmatrix} \in \R^{(n+1) \times (n+1)}\,,
\]
where $ \bar{L}^{(1)} := \begin{bmatrix}
        2(\rho-1)
    \end{bmatrix}$ and
\begin{align*}
    \bar{L}^{(k+1)} &:= \begin{bmatrix}
        \bar{L}^{(k)} + (c^{(k)} - \pi^{(k)})(c^{(k)} - \pi^{(k)})^T & \textcolor{lightgray}{\bm{0}_{n \times 1}} & \textcolor{lightgray}{\bm{0}_{n \times n}} \\
        \textcolor{lightgray}{\bm{0}_{1 \times n}} & \textcolor{lightgray}{0} & \textcolor{lightgray}{\bm{0}_{1 \times n}}\\
        \textcolor{lightgray}{\bm{0}_{n \times n}} & \textcolor{lightgray}{\bm{0}_{n \times 1}} & \rho^2 (\bar{L}^{(k)} + (c^{(k)} - \pi^{(k)})(c^{(k)} - \pi^{(k)})^T)
    \end{bmatrix} \\
    &\quad+ \begin{bmatrix}
    \textcolor{lightgray}{\bm{0}_{n \times n}} & -\rho^k(c^{(k)} - \pi^{(k)}) & \textcolor{lightgray}{\bm{0}_{n \times n}} \\
    -\rho^k(c^{(k)} - \pi^{(k)})^T & (\rho^{k-1}+1)(\rho^{k+1}+1) & -\rho (\pi^{(k)})^T \\
    \textcolor{lightgray}{\bm{0}_{n \times n}} & -\rho \pi^{(k)} & \textcolor{lightgray}{\bm{0}_{n \times n}} \\
    \end{bmatrix} \\
    &\quad- (c^{(k+1)} - \pi^{(k+1)})(c^{(k+1)} - \pi^{(k+1)})^T\,.
\end{align*}
Also, define
$S^{(k)} := \begin{bmatrix}
        \frac{1}{\sqrt{2}} & (-e_1 + e_{n+1})^T \\
        -e_1 + e_{n+1} & L^{(k)}
    \end{bmatrix} \in \R^{(n+2) \times (n+2)}
$.
\end{definition}

The following lemma proves several properties of these matrices, culminating in the fact that $S^{(k)}$ is positive semidefinite. The main idea is to observe that $L^{(k)}$ is Laplacian (by the recursive construction), and then use that to analyze $S^{(k)}$ by taking a Schur complement.

\begin{lemma}[Properties of $L^{(k)}$ and $S^{(k)}$]\label{lem:laplacian} 
    The following hold for all $k \in \N$.
    \begin{enumerate}[label=(\alph*)]
    \item All nondiagonal entries of $\bar{L}^{(k)} + (c^{(k)} - \pi^{(k)})(c^{(k)} - \pi^{(k)})^T$ are nonpositive (and thus similarly for $\bar{L}^{(k)}$ and $L^{(k)}$).  
    \item $L^{(k)}$ is Laplacian.
    \item $S^{(k)}$ is positive semidefinite.
\end{enumerate}
\end{lemma}

\begin{proof}
We prove the items in order. \underline{Proof of (a).} This follows by induction. The base case $k=1$ is clear by inspection. The induction step follows from the recursive definition of $\bar{L}^{(k)}$ and the fact that $c^{(k)} \geq \pi^{(k)}$ by Lemma \ref{lem:ck}. \underline{Proof of (b).} By item (a), it suffices to show that each row sum of $L^{(k)}$ is 0. Note that for any $k$, the sum of the final row is $-\sum c^{(k)} + 2(\rho^k-1) = 0$ by Lemma \ref{lem:ck}. The other rows also sum to zero, as can be checked via a straightforward calculation from the recursive construction; details in Appendix \ref{appendix:laplacian}. \underline{Proof of (c).} Since the top-left entry of $S^{(k)}$ is positive, it suffices to check that the corresponding Schur complement $L^{(k)} - \sqrt{2}(-e_1 + e_{n+1})(-e_1 + e_{n+1})^T$ is positive semidefinite. In particular, we show that this matrix is Laplacian. Since both $L^{(k)}$ and $\sqrt{2}(-e_1 + e_{n+1})(-e_1 + e_{n+1})^T$ are Laplacian by item (b), each row sum of the Schur complement is 0. For the nondiagonal entries, it suffices to check that the $(1, n+1)$ entry of the Schur complement is nonpositive; this is straightforward since $c^{(k)}_1 \geq \sqrt{2}$ for all $k$.
\end{proof}

\subsection{Verification of the multi-step descent identity}\label{ssec:verification}

Here we provide a proof sketch of item (iii) of Theorem~\ref{thm:maincert}; full details are deferred to Appendix~\ref{app:verification}. Item (iii) requires verifying the multi-step descent identity~\eqref{eq:maincert}, which amounts to checking that the coefficients match on both sides of the identity. This identity has two components: a linear form (in the function evaluations $\{f_i\}$ and $\{h_i\}$), plus a quadratic form (in the gradient evaluations $\{g_i\}$, subgradient evaluations $\{s_i\}$, and initial distance $x_0 - x_*$). Thus, it suffices to separately check that the coefficients match for the linear and quadratic forms.

\par This is simple for the linear form. The coefficients in $\{f_i\}$ match since they are linear combinations of only the multipliers $\{\lambda_{i, j}\}$. As they are identical to those for the analysis of vanilla GD, we can appeal to existing results. 
The coefficients in $\{h_i\}$ can be seen to match by simply expanding and collecting terms from the definition of co-coercivities; see Appendix \ref{appendix:linear} for details.

\par For the quadratic form, verifying the identity requires checking that $\Theta(n^2)$ coefficients match, for $n = 2^{k}-1$. A key observation that makes this verification tractable is that if we combine terms in a certain order, then both sides of the identity become ``succinct'' quadratic polynomials in $O(1)$ variables; in particular, it suffices to check that the corresponding $O(1)$ coefficients match. This conciseness is crucially due to the construction of the multipliers based on the recursive gluing technique. The proof therefore amounts to 1) computing the \emph{constant} number of coefficients of these variables on both sides of the identity, and 2) checking that they match. Both steps are conceptually simple (albeit tedious); details are deferred to Appendices~\ref{appendix:quadsetup} and~\ref{appendix:quadproof}, respectively.

\acks{This work was partially supported by a Seed Grant Award from Apple.}

\bibliography{ref}

\appendix

\newpage
\appendix

\section{Deferred details}

\subsection{Illustration of the recursive gluing}\label{appendix:figure}

Here we provide a visual representation of the individual components of the multipliers defined in \S\ref{ssec:multipliers} by recursive gluing.

\begin{figure}[H]

\centering

  \subfigure[\footnotesize $\lambda^{(k+1)}$]{\includegraphics[width=0.45\textwidth]{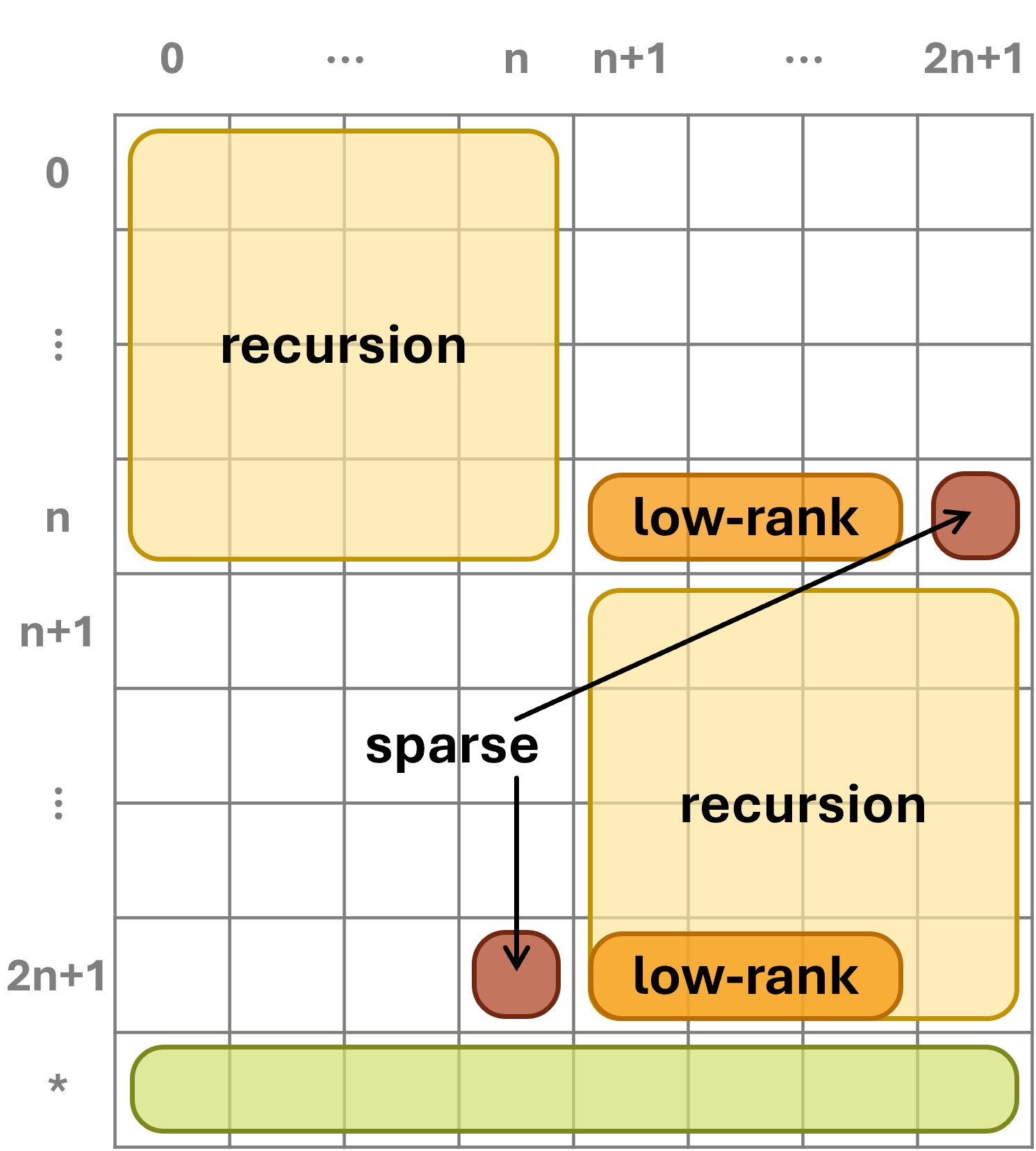}
    \label{fig:lambda}}
  \subfigure[\footnotesize $\mu^{(k+1)}$]{\includegraphics[width=0.4\textwidth]{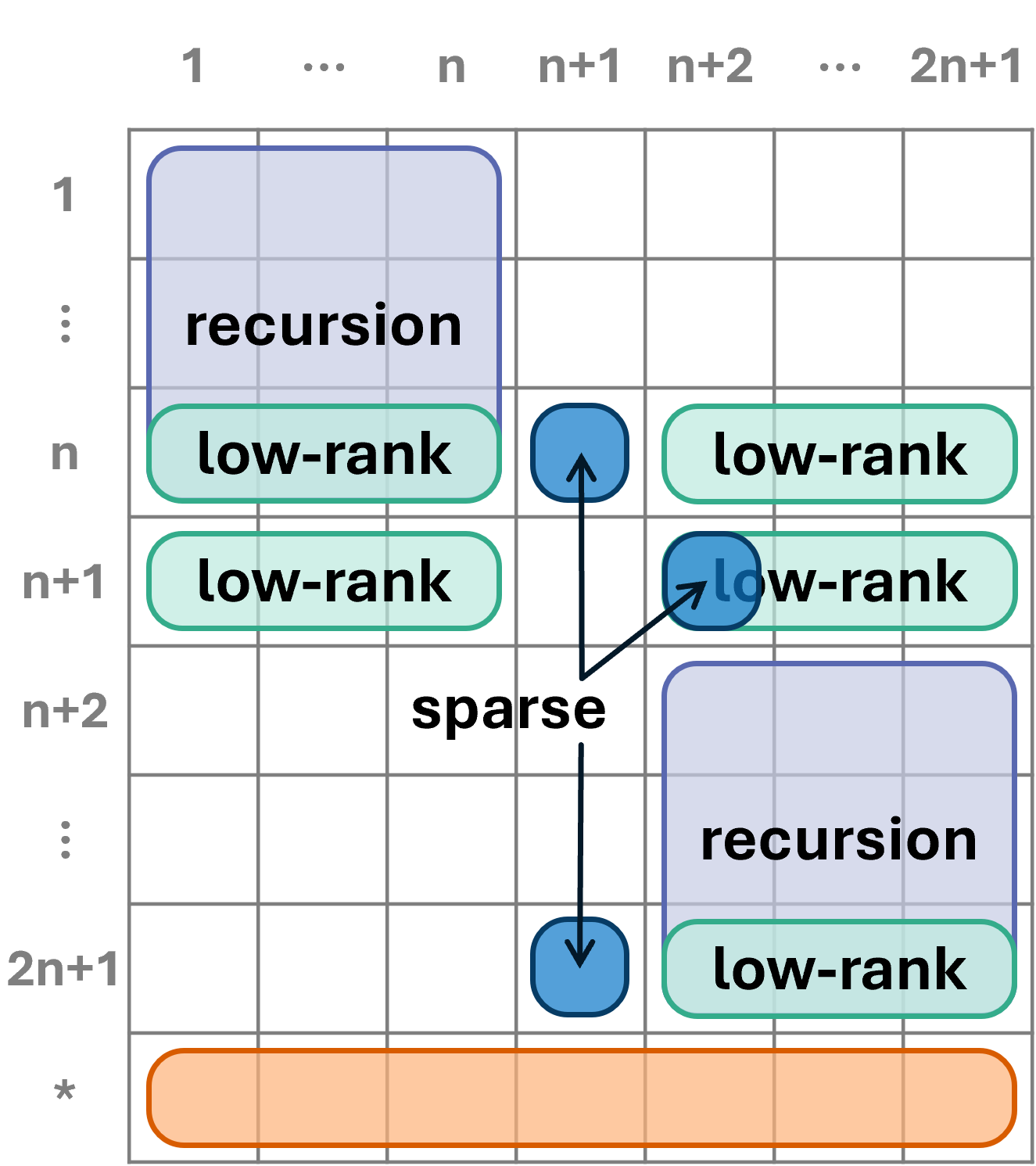}
    \label{fig:mu}}
    
   \caption{\footnotesize Illustration of the structure of $\lambda^{(k+1)}$ (left) and $\mu^{(k+1)}$ (right). Both are constructed via the \emph{recursive gluing} technique, which combines three components. The recursion component glues two copies from $\lambda^{(k)}$ or $\mu^{(k)}$, respectively. The sparse component consists of $O(1)$ nonzero entries. The low-rank component consists of linear combinations of $\pi^{(k)}$ and $c^{(k)}$. Intuitively, the latter two components enable ``gluing'' the two recursive components by controlling the long-range effect of the large step in the middle. Altogether, this allows inductively establishing the desired convergence rate for $2n+1 = 2^{k+1}-1$ given $n = 2^k-1$.}
   \label{fig:cert}
\end{figure}

\subsection{Tightness 
of the silver convergence rate
}\label{appendix:lowerbound}

    Theorem \ref{thm:main} is asymptotically tight in that the exponent $\log_2 \rho$ in the rate $O(n^{-\log_2 \rho})$ is optimal for proximal GD with the silver stepsize schedule. Moreover, the hidden constant in the asymptotic rate is tight up to a modest constant. Indeed, there exists a problem instance such that 
    \begin{align}
        F(x_n) - F(x_*) = \frac{M}{4\rho^k - 4}\norm{x_0 - x_*}^2\,,
        \label{eq:tight}
    \end{align}
    which matches the aforementioned exponent since $\rho^k = \rho^{\log_2 (n+1)} \asymp \rho^{\log_2 n} = n^{\log_2 \rho}$. This bound~\eqref{eq:tight} is achieved by the following $1$-dimensional problem (normalized with $M=1$ and $\norm{x_0 - x_*}^2 = 1$ without loss of generality) from \citet[Theorem 2.9]{drorithesis} which is often used for lower bounds on proximal GD. Consider
    \[
        \min_{x \in \R} F(x) = f(x) + h(x)\,,
    \]
    where $f(x) = ax$ is a linear function with slope $a = \frac{1}{2(\rho^k - 1)}$, and $h(x) = \iota\{x \geq 0\}$ is the indicator function for $x \geq 0$. Observe that $x_* = 0$, and that from initialization $x_0 = 1$, the final iterate of proximal GD is $x_n = 1 - a(\rho^k - 1)$, implying $F(x_n) - F(x_*) = a (1 - a(\rho^k - 1))$. Maximizing this rate over the slope $a \geq 0$ gives the aforementioned value of $a$ and the rate~\eqref{eq:tight}.

    \par We make three remarks about this ``hard'' problem instance. First, since $h$ is the indicator of a convex set, this construction can also be viewed as a hard problem $\min_{x \geq 0} f(x)$ for projected GD. Second, by running semidefinite programs~\citep{dt14, thg17} to numerically estimate the worst-case convergence rate of proximal GD with the silver stepsizes, it appears that~\eqref{eq:tight}
    is the \emph{exactly} tight bound for this algorithm. Third, we compare the constants in this rate~\eqref{eq:tight} for proximal GD with the silver stepsizes, to the exact rate $f(x_n) - f(x_*) \leq \frac{M}{4\rho^k - 2}\norm{x_0 - x_*}^2$ of vanilla GD with the silver stepsizes~\citep{rppa}. The former is achieved by a linear function (equivalently, the $\ell_1$ norm) constrained to the set $\{x \geq 0\}$, whereas the latter is achieved by a certain Huber loss function (i.e., the Moreau envelope of the $\ell_1$ norm) over $\R$. This minor difference in problem instances results in the minor difference between the denominators in the rates: $4\rho^k - 4$ versus $4 \rho^k - 2$. This discrepancy commonly arises between an algorithm and its proximal counterpart; see, for example, the discussion after \citet[Theorem 1]{optista}.

\subsection{Nonnegativity of the multipliers}\label{app:nonnegativity}

Here we complete the proof of Lemma \ref{lem:multpositive} by verifying the nonnegativity of the $n-1$ entries $\{\mu^{(k+1)}_{n, j}\}_{1 \leq j < n}$. To make the induction step clear, we isolate the claim into its own lemma.

\begin{lemma}\label{lem:posmultnontrivial}
For $k \in \N$, let $\mu^{(k)}$ be as in Definition \ref{def:mu} and $n = 2^k-1$. Then for all $1 \leq j < n$,
\[\bmu^{(k)}_{n, j} = (\rho^k-1)(c^{(k)}_j - \pi^{(k)}_j)\,.
\]
As a corollary, for all $1 \leq j < n$,
\[\mu^{(k+1)}_{n, j} = \bmu^{(k+1), \rec}_{n, j} + \bmu^{(k+1), \lr}_{n, j} =  \bmu^{(k)}_{n, j} + (1-\frac{\rho^k}{\rho^{k-1} + 1})(c^{(k)}_j - \pi^{(k)}_j) = \frac{\rho^{2k-1}}{\rho^{k-1}+1}(c^{(k)}_j - \pi^{(k)}_j) \geq 0\,.
\]
\end{lemma}
\begin{proof}
    We prove by induction. For the base case, note that for $k = 1$ no such entry exists; and for $k = 2$, we have $\bmu^{(2)}_{3, 1} = 0$ and $\bmu^{(2)}_{3, 2} = (\rho-\tfrac{1}{\rho})(\rho^0+1) = (\rho^2-1)\tfrac{\rho^0+1}{\rho} = (\rho^2-1)(c^{(2)}_2 - \pi^{(2)}_2)$.
    
    Assuming that the result holds for $k \geq 2$, by the recursive definition we have $\bmu^{(k+1)}_{2n+1, j} = 0$ for all $1 \leq j \leq n$ and $\bmu^{(k+1)}_{2n+1, n+1} = (\rho-\tfrac{1}{\rho^k})(\rho^{k-1}+1) = (\rho^{k+1}-1)\tfrac{\rho^{k-1}+1}{\rho^k} = (\rho^{k+1}-1)(c^{(k+1)}_{n+1} - \pi^{(k+1)}_{n+1})$. For $n+2 \leq j < 2n+1$,
    \begin{align*}
        \bmu^{(k+1)}_{2n+1, j} &= \rho^2 \bmu^{(k+1)}_{n, j-n-1} + (\rho+1)c^{(k)}_{j-n-1} - (1+\frac{1}{\rho^k})\pi^{(k)}_{j-n-1} \\
        &= \rho^2(\rho^k-1)(c^{(k)}_{j-n-1} - \pi^{(k)}_{j-n-1}) + (\rho+1)c^{(k)}_{j-n-1} - (1+\frac{1}{\rho^k})\pi^{(k)}_{j-n-1} \\
        &= (\rho^{k+1} - 1)(\rho c^{(k)}_{j-n-1} - (\rho - \frac{1}{\rho^k})\pi^{(k)}_{j-n-1})\\
        &= (\rho^{k+1}-1)(c^{(k+1)}_j - \pi^{(k+1)}_j)\,.
    \end{align*}
\end{proof}

\subsection{Laplacian structure}\label{appendix:laplacian}

Here we complete the proof of Lemma \ref{lem:laplacian} by showing that each row sum of $L^{(k)}$ is $0$, for all $k \in \N$. We use induction. The base case $k = 1$ is trivial. Now assume that each row sum of $L^{(k)}$ is 0 (which implies that the $j$th row sum of $\bar{L}^{(k)}$ is $c^{(k)}_j$), and consider the row sums of $L^{(k+1)}$. For this, we make repeated use of the definitions of $\pi^{(k)}$ and $c^{(k)}$ and their respective helper lemmas (Lemmas~\ref{lem:silver} and~\ref{lem:ck}) from \S\ref{ssec:helper}.

For $1 \leq j \leq n$, the $j$th row sum of $L^{(k+1)}$ is given by
\begin{align*}
    c^{(k)}_j + (\sum(c^{(k)} - \pi^{(k)}))(c^{(k)}_j - \pi^{(k)}_j) - \rho^k(c^{(k)}_j - \pi^{(k)}_j) - c^{(k+1)}_j = 0\,,
\end{align*}
from $\sum (c^{(k)} - \pi^{(k)}) = \rho^k-1$ and $c^{(k+1)}_j = \pi^{(k+1)}_j = \pi^{(k)}_j$.

The $(n+1)$th row sum of $L^{(k+1)}$  is given by
    \[-\rho^k(\rho^k-1) + (\rho^{k-1}+1)(\rho^{k+1}+1) - \rho(\rho^k-1) - (\sum(c^{(k+1)} - \pi^{(k+1)}))(c^{(k+1)}_{n+1} - \pi^{(k+1)}_{n+1}) - c^{(k+1)}_{n+1} = 0\,,
    \]
from $\sum(c^{(k+1)} - \pi^{(k+1)}) = \rho^{k+1} - 1$, $c^{(k+1)}_{n+1} = (1+\frac{1}{\rho^k})(\rho^{k-1}+1)$ and $\pi^{(k+1)}_{n+1} = \rho^{k-1}+1$.

Finally, for $n+2 \leq j \leq 2n+1$, the $j$th row sum of $L^{(k+1)}$ is given by
    \begin{align*}
        &\rho^2 (c_{j-n-1}^{(k)} + (\rho^k-1)(c_{j-n-1}^{(k)} - \pi^{(k)}_{j-n-1})) - \rho \pi^{(k)}_{j-n-1}\\
        &\quad- (\sum(c^{(k+1)} - \pi^{(k+1)}))(c^{(k+1)}_j - \pi^{(k+1)}_j) - c^{(k+1)}_j = 0\,,
    \end{align*}
from $c^{(k+1)}_j = \rho c^{(k)}_{j-n-1} - (\rho - 1 - \frac{1}{\rho^k})\pi^{(k)}_{j-n-1}$.

\section{Verification of the multi-step descent identity}\label{app:verification}

Here we prove item (iii) of Theorem~\ref{thm:maincert}. The base case $k=1$ is in Appendix \ref{appendix:basecase}; we then prove the identity for larger $k$ by induction. As described in the overview in \S\ref{ssec:verification}, the desired identity has both a linear and a quadratic component, and it suffices to check these separately. We begin by introducing helpful bookkeeping notation in Appendix \ref{appendix:decomp}, and then we check the linear component in Appendix \ref{appendix:linear} and the quadratic components in Appendices~\ref{appendix:quadsetup} and~\ref{appendix:quadproof}. All proofs begin with a straightforward inspection over the coefficients of the relevant terms after expanding the definition of the co-coercivities; for the convenience of the reader, we explicitly write these expanded forms in Proposition \ref{prop:decompdetail}. The proofs then check that the constant number of coefficients match in a conceptually straightforward (albeit tedious) manner.

Throughout this section, for notational shorthand we denote $c = c^{(k)}$ and $S = S^{(k)}$. In the few cases when needed, we explicitly write $c^{(k+1)}$ and $S^{(k+1)}$ to distinguish them respectively from $c$ and $S$.

\subsection{Base case}\label{appendix:basecase}

Here we prove the base case $k = 1$ for the multi-step descent identity~\eqref{eq:maincert}. For $k=1$,
    \[\lambda = \ \begin{bNiceArray}{cc}[first-row, first-col]
        & 0 & 1 \\
        0 & 0 & \rho \\
        1 & 1 & 0 \\
        * & \rho - 1 & \rho
    \end{bNiceArray}, \quad 
    \mu = \ \begin{bNiceArray}{c}[first-row, first-col]
        & 1 \\
        1 & 0 \\
        * & 2\rho - 1
    \end{bNiceArray} \,,
    \quad S = \begin{bmatrix}
        \frac{1}{\sqrt{2}} & 1 & -1 \\ 
        -1 & 2(\rho-1) & -2(\rho-1) \\
        1 & -2(\rho-1) & 2(\rho-1)
    \end{bmatrix}\,,
    \]
    $u = x_0 - x_* - (\rho-1)g_0 - \rho g_1 - 2(\rho-1)s_1 - s_*$, and $V = \begin{bmatrix}
        x_0 - x_* \ | \ s_1 \ | \ s_*
    \end{bmatrix}$. Note that $x_1 = x_0 - (\rho-1)(g_0 + s_1)$ and $g_* = -s_*$ by the definition of proximal GD. By plugging in these values and expanding the definition of the co-coercivities (Definition~\ref{def:coco}), the multi-step descent identity \eqref{eq:maincert} for $k=1$ amounts to the identity
    \begin{align*}
        &\rho(f_0 - f_1 - \angs{g_1}{x_0 - x_1}-\frac{1}{2}\norm{g_0 - g_1}^2) + (f_1 - f_0 - \angs{g_0}{x_1-x_0} - \frac{1}{2}\norm{g_1 - g_0}^2) \\
        &\quad+ (\rho-1)(f_* - f_0 - \angs{g_0}{x_* - x_0} - \frac{1}{2}\norm{g_* - g_0}^2) \\
        &\quad+ \rho(f_* - f_1 - \angs{g_1}{x_* - x_1} - \frac{1}{2}\norm{g_* - g_1}^2) + (2\rho-1)(h_* - h_1 - \angs{s_1}{x_* - x_1}) \\
        &= -(2\rho-1)(f_1 + h_1 - f_* - h_*) + \frac{\rho}{2\sqrt{2}}\norm{x_0 - x_*}^2 \\
        &\quad- \frac{1}{2}\|x_0 - x_* - (\rho-1)g_0 - \rho g_1 - 2(\rho-1)s_1 - s_*\|^2 \\
        &\quad- \frac{1}{4\sqrt{2}}\norm{x_0 - x_*}^2 - \frac{1}{4\sqrt{2}}\norm{x_0 - x_* - 2(\rho-1)(s_1 - s_*)}^2\,.
    \end{align*}
    It is straightforward to check this identity by matching coefficients. For example, on the left hand side, the coefficient of $f_1$ is $-\rho + 1 - \rho = -(2\rho - 1)$ and the coefficient of $\angs{g_0}{s_1}$ is $(\rho-1)-(\rho-1)(2\rho-1) = -2(\rho-1)^2$, respectively matching the corresponding coefficients on the right hand side.

\subsection{Bookkeeping}\label{appendix:decomp}

We decompose terms in the multi-step descent identity based on the recursive definition of the multipliers. For the convenience of the reader, expanded expressions for these terms are provided in Proposition \ref{prop:decompdetail}.

\begin{definition}[Bookkeeping decomposition] For $\lambda^{(k+1)}$ as in Definition \ref{def:lambda} and $\mu^{(k+1)}$ as in Definition \ref{def:mu}, let
\begin{equation}\label{eq:TCdecomp}
\begin{split}
    R &:= \sum_{i, j} \blambda^{(k+1), \rec}_{i, j}Q_{ij}^f + \sum_{i, j} \bmu^{(k+1), \rec}_{i, j}Q_{ij}^h\,, \\
    T^f &:= \sum_{j} \ulambda_{j}^{(k+1)} Q_{*j}^f\,, \\
    T^h &:= \sum_{j} \umu_{j}^{(k+1)}Q_{*j}^h\,, \\
    C^f &:= \sum_{i, j} (\blambda^{(k+1), \sparse}_{i, j} + \blambda^{(k+1), \lr}_{i, j})Q_{ij}^f\,, \\
    C^h &:= \sum_{i, j} (\bmu^{(k+1), \sparse}_{i, j} + \bmu^{(k+1), \lr}_{i, j})Q_{ij}^h\,. 
\end{split}
\end{equation}
\end{definition}

The induction step from $k$ to $k+1$ that we prove can be formally stated as follows. The rest of the section is devoted to proving this identity, by certifying that the coefficients on both sides are identical. Combined with the base case (Appendix \ref{appendix:basecase}), this directly proves the multi-step identity (item (iii) of Theorem \ref{thm:maincert}).

\begin{theorem}\label{thm:wts} Let $R, T^f, T^h, C^f, C^h$ be as in \eqref{eq:TCdecomp}, and assume that \eqref{eq:maincert} holds for $k$. Then
\begin{equation}\label{eq:wts}
\begin{split}
    &R + T^f + T^h + C^f + C^h \\
    &\quad= (2\rho^{k+1} - 1) (F_* - F_{2n+1}) + \frac{\rho}{2\sqrt{2}}\norm{x_0 - x_*}^2 - \frac{1}{2}(\norm{\widehat{u}}^2 + \Tr(\widehat{V}S^{(k+1)}\widehat{V}^T))
\end{split}
\end{equation}
where $\widehat{u} := x_0 - x_* - \sum_{i=0}^{2n} \alpha_i g_i - \rho^{k+1} g_{2n+1} - \sum_{j=1}^{2n+1} c^{(k+1)}_j s_j - s_*$ and \\$\widehat{V} := \begin{bmatrix}
    x_0 - x_* \ | \ s_1 \ | \ \dots \ | \ s_{2n+1} \ | \ s_*
\end{bmatrix}$.
\end{theorem}

It is convenient to adopt the notation of \citet{gsw24} to write $\dbsmall{\cdot}_f$ to denote terms corresponding to $\{f_i\}$ in a given expression. For example, $\dbsmall{Q_{01}^f}_f = f_0 - f_1$; $\dbsmall{\cdot}_h$ is similarly defined.

These definitions readily extend to the quadratic form. Due to the composite nature of our problem, there are multiple types of vectors we need to consider. We categorize them into four types: $\{g, s, i, o\}$, where each denotes\footnote{While $x_0 - x_*$ (type $i$) or $s_*$ (type $o$) are only single vectors and thus can be subsumed into other types, we treat them separately because the patterns they exhibit in the equations are different from (and simpler than) those for other types.}
\begin{align*}
    g&: g_0, g_1, \dots, g_{2n+1} \text{ (\underline{g}radients corresponding to $f$)}\,,\\
    s&: s_1, s_2, \dots, s_{2n+1} \text{ (\underline{s}ubgradients corresponding to $h$)}\,,\\
    i&: x_0 - x_* \text{ (\underline{i}nitial distance)}\,,\\
    o&: s_* = -g_* \text{ (gradient at \underline{o}ptimality)}\,.
\end{align*}
For $y, z \in \{g, s, i, o\}$ we use $\dbsmall{\cdot}_{y, z} = \dbsmall{\cdot}_{z, y}$ to denote terms corresponding to the inner products between vectors in type $y$ and $z$. For example, $\dbsmall{Q_{01}^f}_{g, s} = -\sqrt{2}\angs{g_1}{s_1}$. It is clear that $\dbsmall{\cdot}$ (with any subscript) is linear with respect to its input.

For each coefficient of interest, it is easy to see that only certain summands contribute. This observation is formalized as follows and is used throughout the rest of the proof when evaluating coefficients.

\begin{observation}\label{obs:contribclearout} The following values are 0:
\begin{gather*}
    \db{T^h + C^h}_f, \db{T^f + C^f}_h, \db{T^h + C^h}_{g, g}, \db{T^f + C^f}_{s, s}, \\
    \db{T^h}_{i, g}, \db{T^f}_{i, s}, \db{C^f + C^h}_{i, y}, \db{T^h + C^f + C^h}_{o, y}\,,
\end{gather*}
where $y \in \{g, s, i, o\}$.
\end{observation}

\subsection{Linear form}\label{appendix:linear}

Here we verify the linear form in the multi-step descent identity~\eqref{eq:wts}. 

\begin{proposition}[Linear form]\label{prop:linear} Consider the setting of Theorem \ref{thm:wts}. Then
\begin{align*}
    \db{R+ T^f + C^f}_f &= (2\rho^{k+1}-1)(f_* - f_{2n+1})\,, \\
    \db{R + T^h + C^h}_h &= (2\rho^{k+1}-1)(h_* - h_{2n+1})\,.
\end{align*}
\end{proposition}
\begin{proof} The first equality was proved in \citet[Theorem 5.2]{rppa}. For the second equality, expanding each term following the definition yields (see Proposition \ref{prop:decompdetail} for corresponding expressions)
\begin{align*}
        &\db{R + T^h + C^h}_h \\
        &= h_1 - h_n + \sum_{j=1}^n c_j(h_j - h_n) + \rho^2(h_{n+2} - h_{2n+1}) + \sum_{j=n+2}^{2n+1}\rho^2c_{j-n-1}(h_j - h_{2n+1}) \\
        &\quad+ h_* - h_1 + \sum_{j=1}^n \alpha_{j-1}(h_* - h_j)  + (1+\frac{1}{\rho^k})(\rho^{k-1}+1)(h_* - h_{n+1})  \\
        &\quad+ \sum_{j=n+2}^{2n+1}(\rho(c_{j-n-1}-\alpha_{j-1}) + (1+\frac{1}{\rho^k})\alpha_{j-1})(h_* - h_j)  + \rho^k(h_n - h_{n+1}) + \rho^2(h_{n+1}-h_{n+2}) \\
        &\quad+ (\rho-\frac{1}{\rho^k})(\rho^{k-1}+1)(h_{2n+1} - h_{n+1}) + \sum_{j=1}^n(c_j - \alpha_{j-1})(h_n - h_j)  \\
        &\quad+ \cancel{\sum_{j=1}^n \frac{\rho^k}{\rho^{k-1}+1}(c_j - \alpha_{j-1})(h_{n+1}-h_n)} + \sum_{j=n+2}^{2n+1} \rho \alpha_{j-1} (h_{n+1} - h_{2n+1}) \\
        &\quad+ \cancel{\sum_{j=n+2}^{2n+1} \frac{\rho^k}{\rho^{k-1}+1}\alpha_{j-1}(h_n - h_{n+1})} + \sum_{j=n+2}^{2n+1}((\rho+1)c_{j-n-1} + (\rho-1-\frac{1}{\rho^k})\alpha_{j-1})(h_{2n+1}-h_j)\,.
\end{align*}
The cancellation in the middle is due to $\sum_{j=1}^n c_j = 2(\rho^k - 1) = 2\sum_{j=1}^n \alpha_{j-1}$ (Lemmas \ref{lem:silver} and \ref{lem:ck}). By collecting the coefficients for each $h_j$, this is equal to
\begin{align*}
        &\sum_{j=1}^n(c_j - \alpha_{j-1} - (c_j - \alpha_{j-1}))h_j\\
        &\quad+ \sum_{j=n+2}^{2n+1}(\rho^2 c_{j-n-1} - \rho(c_{j-n-1}-\alpha_{j-1}) - (1+\frac{1}{\rho^k})\alpha_{j-1} - (\rho+1)c_{j-n-1}-(\rho -1-\frac{1}{\rho^k})\alpha_{j-1})h_j \\
        &\quad+ h_1 - h_1 +(-1 -2(\rho^k-1) + \rho^k + (\rho^k - 1))h_n \\
        &\quad+(-(1+\frac{1}{\rho^k})(\rho^{k-1}+1)-\rho^k+\rho^2-(\rho - \frac{1}{\rho^k})(\rho^{k-1}+1)+\rho(\rho^k - 1))h_{n+1} + (\rho^2 - \rho^2)h_{n+2} \\
        &\quad+ (-\rho^2 -2\rho^2(\rho^k-1)+(\rho-\frac{1}{\rho^k})(\rho^{k-1}+1)-\rho(\rho^k-1)+2(\rho+1)(\rho^k-1)+(\rho-1-\frac{1}{\rho^k})(\rho^k-1))h_{2n+1} \\
        &\quad+ (1 + 2(\rho^{k+1} - 1))h_* \\
        &= (2\rho^{k+1}-1)(h_* - h_{2n+1})\,,
    \end{align*}
    where the coefficient of $h_*$ is calculated using $\sum_{j=1}^n \alpha_{j-1} + \sum_{j=n+2}^{2n+1} (\rho(c_{j-n-1} - \alpha_{j-1}) + (1+\frac{1}{\rho^k})\alpha_{j-1}) = \sum c^{(k+1)} = 2(\rho^{k+1}-1)$.
\end{proof}

\subsection{Quadratic form: setup}\label{appendix:quadsetup}

As briefly mentioned in the overview in \S\ref{ssec:verification}, instead of checking the $\Theta(n^2)$ coefficients corresponding to all inner products between the $\Theta(n)$ variables $\{g_i\}, \{s_i\}$, and $x_0 - x_*$, we simplify the analysis by first observing that the quadratic form expressions in \eqref{eq:TCdecomp} can be expressed succinctly in a constant number of modified variables (obtained by taking linear combinations of the original variables). The upshot is that then we need to only compute and verify a constant number of coefficients. These modified variables are as follows.

\begin{definition}[Simple vectors for the quadratic form]\label{def:qfvectors} Define
\begin{align*}
    \gamma &= \begin{bmatrix} \gamma_1 \ | \ \gamma_2 \ | \ \gamma_3 \ | \ \gamma_4        
    \end{bmatrix} \\
    &:= \begin{bmatrix}
        \sum_{i=0}^{n-1} \alpha_i g_i \ | \ g_n \ | \ \sum_{i=n+1}^{2n} \alpha_i g_i \ | \ g_{2n+1}
    \end{bmatrix}\,, \\
    \sigma &= \begin{bmatrix}
        \sigma_1 \ | \ \sigma_2 \ | \ \sigma_3 \ | \ \sigma_4 \ | \ \sigma_5
    \end{bmatrix} \\
    &:= \begin{bmatrix}
        \sum_{j=1}^n \alpha_{j-1} s_j \ | \ \sum_{j=1}^n c_j s_j \ | \ s_{n+1} \ | \ \sum_{j=n+2}^{2n+1} \alpha_{j-1} s_j \ | \ \sum_{j=n+2}^{2n+1} c_{j-n-1} s_j
    \end{bmatrix} \,.
\end{align*}
From the recursive properties of $\pi^{(k)}$ and $c^{(k)}$ (Lemma \ref{lem:silver} and Definition \ref{def:ck}), one can observe that for
\begin{align*}
    w_g &:= [1, \rho^{k-1}+1, 1, \rho^{k+1}]\,, \\
    w_s &:= [1, 0, (1+\frac{1}{\rho^k})(\rho^{k-1}+1), -(\rho-1-\frac{1}{\rho^k}), \rho]\,,
\end{align*}
we have
\begin{equation*}
\begin{aligned}
    \gamma w_g &= \sum_{i=0}^{2n}\alpha_i g_i + \rho^{k+1}g_{2n+1}\,, \\
    \sigma w_s &= \sum_{j=1}^{2n+1} c^{(k+1)}_j s_j\,.
\end{aligned}    
\end{equation*}
\end{definition}

These vectors naturally arise from the quadratic forms in \eqref{eq:maincert} and \eqref{eq:wts}, which can be concisely represented by them; in particular, 
\begin{align*}
    \frac{1}{2}\norm{u}^2 &= \frac{1}{2}\norm{x_0 - x_* - \gamma_1 - \rho^k \gamma_2 - \sigma_2 - s_*}^2\,, \\
    \frac{1}{2}\norm{\widehat{u}}^2 &= \frac{1}{2}\norm{x_0 - x_* - \gamma w_g - \sigma w_s - s_*}^2\,,
\end{align*}
and the terms $\frac{1}{2}\Tr(VSV^T)$ and $\frac{1}{2}\Tr(\widehat{V}S^{(k+1)}\widehat{V}^T)$ together can be expressed similarly by $\sigma$ modulo simple operations from recursion. 

While some expressions in \eqref{eq:TCdecomp} can be represented with $\gamma$ or $\sigma$ in a straightforward way, in some cases the terms should be properly combined with others, for which we provide the details. First, we identify particular terms that are sums of inner products. These are categorized as $A_1, A_2, A_3$ in Proposition \ref{prop:decompdetail}, and with a slight abuse of notation we use $A_i$ for the sum of the respective entries as well. The following equalities can be verified by inspection.

\begin{observation}\label{obs:combineforsimpleA} For $A_1, A_2, A_3$ in Proposition \ref{prop:decompdetail} and $\gamma, \sigma$ as in Definition \ref{def:qfvectors},
\begin{align*}
    \db{A_1}_{g, s} &= \angs{\gamma_1}{\sigma_2 - \sigma_1} \,, \\
    \db{A_2}_{g, s} &= \rho \angs{\gamma_3}{2\sigma_1 + (\rho^{k-1}+1)\sigma_3 + \sigma_4} \,, \\
    \db{A_3}_{g, s} &= \angs{\gamma_1 + (\rho^{k-1}+1)\gamma_2 + \gamma_3}{(\rho-1-\frac{1}{\rho^k})\sigma_4 + (\rho+1)\sigma_5} \,,
\end{align*}
and
\begin{align*}
    \db{A_1}_{s, s} &= \angs{\sigma_1}{\sigma_2 - \sigma_1} \,, \\
    \db{A_2}_{s, s} &= 0 \,, \\
    \db{A_3}_{s, s} &= \angs{\sigma_1 + (\rho^{k-1}+1)\sigma_3 + \sigma_4}{(\rho-1-\frac{1}{\rho^k})\sigma_4 + (\rho+1)\sigma_5}\,.
\end{align*}
\end{observation}

Next, we identify other terms ($B_1, B_2, B_3$ in Proposition \ref{prop:decompdetail}, with similar abuse of notation) which are related to specific (sub)gradient vectors; namely, $g_n = \gamma_2, g_{2n+1} = \gamma_4$, and $s_{n+2}$. These can also be verified by inspection and by using the Pell recurrence $\rho^2 = 2\rho + 1$ for the silver ratio $\rho$.
\begin{observation}\label{obs:combineforsimpleB} For $B_1, B_2, B_3$ in Proposition \ref{prop:decompdetail} and $\gamma, \sigma$ as in Definition \ref{def:qfvectors},
\begin{align*}
    \db{B_1}_{g, s} &= \rho^k\angs{\gamma_2}{\sigma_1 + (\rho^{k-1}+1)\sigma_3 + \sigma_4} \,, \\
    \db{B_2}_{g, s} &= \rho^2\angs{\gamma_1 + (\rho^{k-1}+1)\gamma_2}{s_{n+2}} \,, \\
    \db{B_3}_{g, s} &= \rho^k(\rho+1)\angs{\gamma_4}{\sigma_1 + (\rho^{k-1}+1)\sigma_3 + \sigma_4} \,,
\end{align*}
and
\begin{align*}
    \db{B_1}_{s, s} &= 0 \,, \\
    \db{B_2}_{s, s} &= \rho^2\angs{\sigma_1 + (\rho^{k-1}+1)\sigma_3}{s_{n+2}} \,, \\
    \db{B_3}_{s, s} &= 0 \,.
\end{align*}
\end{observation}

\subsection{Quadratic form: verification}\label{appendix:quadproof}
After the preprocessing steps in Appendix \ref{appendix:quadsetup}, we are ready to prove that the coefficients of the quadratic form match in \eqref{eq:wts}. The proof is divided into multiple parts based on the combination of $y, z \in \{g, s, i, o\}$.

First, we verify the terms involving $\angs{g_i}{g_j}$ in the multi-step descent identity. This is identical to (and thus follows from) the corresponding analysis of vanilla GD.

\begin{proposition}[Quadratic form in $(g, g)$; {\citet[Theorem 5.2]{rppa}}]\label{prop:gg} Consider the setting of Theorem \ref{thm:wts} and let $\gamma, w_g$ be as in Definition \ref{def:qfvectors}. Then
\begin{align*}
    \db{R + T^f + C^f}_{g, g} &= -\frac{1}{2}\norm{\gamma w_s}^2\,.
\end{align*}
\end{proposition}

The next two propositions respectively verify the terms involving $\angs{g_i}{s_j}$ and $\angs{s_i}{s_j}$. 

\begin{proposition}[Quadratic form in $(g, s)$]\label{prop:gs} Consider the setting of Theorem \ref{thm:wts} and let $\gamma, w_g, \sigma, w_s$ be as in Definition \ref{def:qfvectors}. Then
\begin{align*}
    \db{R + T^f + T^h + C^f + C^h}_{g, s} &= -\angs{\gamma w_g}{\sigma w_s}\,.
\end{align*}
\end{proposition}
\begin{proof}
    Using the representation with respect to $\gamma$ and $\sigma$ given from Observation \ref{obs:combineforsimpleA} and Observation \ref{obs:combineforsimpleB},
    \begin{align*}
        &\db{R + T^f + T^h + C^f + C^h}_{g, s} \\
        &= \db{A_1 + A_2 + A_3 + B_1 + B_2 + B_3}_{g, s} + \db{\text{(remaining terms in \eqref{eq:TCdecomp})}}_{g, s} \\
        &= \angs{\gamma_1}{\sigma_2 - \sigma_1} + \rho \angs{\gamma_3}{2\sigma_1 + (\rho^{k-1}+1)\sigma_3 + \sigma_4} \\
        &\quad+ \angs{\gamma_1 + (\rho^{k-1}+1)\gamma_2 + \gamma_3}{(\rho-1-\frac{1}{\rho^k})\sigma_4 + (\rho+1)\sigma_5} + \rho^k\angs{\gamma_2}{\sigma_1 + (\rho^{k-1}+1)\sigma_3 + \sigma_4} \\
        &\quad + \cancel{\rho^2\angs{\gamma_1 + (\rho^{k-1}+1)\gamma_2}{s_{n+2}}} + \rho^k(\rho+1)\angs{\gamma_4}{\sigma_1 + (\rho^{k-1}+1)\sigma_3 + \sigma_4} - \angs{\gamma_1 + \rho^k \gamma_2}{\sigma_2} \\
        &\quad  - \rho^2\angs{\gamma_1 + (\rho^{k-1}+1)\gamma_2}{\sigma_5} - \rho^2\angs{\gamma_3 + \rho^k \gamma_4}{\sigma_1 + (\rho^{k-1} + 1)\sigma_3} - \rho^2\angs{\gamma_3 + \rho^k \gamma_4}{\gamma_5} \\
        &\quad- \cancel{\rho^2\angs{\gamma_1 + (\rho^{k-1} + 1)\gamma_2}{s_{n+2}}} - (\rho^{k-1}+1)\angs{\gamma_2}{\sigma_1} \\
        &\quad- (1+\frac{1}{\rho^k})(\rho^{k-1}+1)\angs{\gamma_1 + (\rho^{k-1}+1)\gamma_2}{\sigma_3} - \rho\angs{\gamma_4}{(\rho^{k-1}+1)\sigma_3 + \sigma_4} \\
        &\quad- \rho^k(\rho^{k-1}+1)\angs{\gamma_2}{\sigma_3} + (\rho-\frac{1}{\rho^k})(\rho^{k-1}+1)\angs{\gamma_3}{\sigma_3} + \rho^k\angs{\gamma_2}{\sigma_2 - \sigma_1}- \rho\angs{\gamma_3}{\sigma_4} \\
        &\quad- \rho^k\angs{\gamma_2}{\sigma_4}\,.
    \end{align*}
    This is equal to $-\angs{\gamma w_g}{\sigma w_s}$. This can be checked by hand; for brevity, we provide a simple Mathematica script that rigorously verifies these identities in the supplementary file~\citep{MathematicaURL}.
\end{proof}
\begin{proposition}[Quadratic form in $(s, s)$]\label{prop:ss} Consider the setting of Theorem \ref{thm:wts} and let $\sigma, w_s$ be as in Definition \ref{def:qfvectors}. Then
\[\db{R + T^h + C^h}_{s, s} = -\frac{1}{2}\norm{\sigma w_s}^2 - \db{\frac{1}{2}\Tr(\widehat{V} S^{(k+1)} \widehat{V}^T)}_{s, s}\,,
\]
where $S^{(k+1)}$ is as in Definition \ref{def:laplacian} and $\widehat{V} = \begin{bmatrix}
        x_0 - x_* \ | \ s_1 \ | \ \dots \ | \ s_{2n+1} \ | \ s_*
    \end{bmatrix}$.
\end{proposition}
\begin{proof} With the expressions from Observations \ref{obs:combineforsimpleA} and \ref{obs:combineforsimpleB},
\begin{align*}
    &\db{R + T^h + C^h + \frac{1}{2}\Tr(\widehat{V} S^{(k+1)} \widehat{V}^T)}_{s, s} \\
    &= \db{A_1 + A_2 + A_3 + B_1 + B_2 + B_3}_{s, s} + \db{\text{(remaining terms in \eqref{eq:TCdecomp})}}_{s, s} + \db{\frac{1}{2}\Tr(\widehat{V} S^{(k+1)} \widehat{V}^T)}_{s, s} \\
    &= \angs{\sigma_1}{\sigma_2 - \sigma_1} + \angs{\sigma_1 + (\rho^{k-1}+1)\sigma_3 + \sigma_4}{(\rho-1-\frac{1}{\rho^k})\sigma_4 + (\rho+1)\sigma_5} \\
    &\quad+ \rho^2\angs{\sigma_1 + (\rho^{k-1}+1)\sigma_3}{s_{n+2}} - \frac{1}{2}\norm{\sigma_2}^2 -\db{\frac{1}{2}\Tr(V S V^T)}_{s, s} \\
    &\quad -\rho^2 \angs{\sigma_1 + (\rho^{k-1}+1)\sigma_3}{\sigma_5} - \frac{\rho^2}{2}\norm{\sigma_5}^2 - \db{\frac{\rho^2}{2}\Tr(\widetilde{V} S \widetilde{V}^T)}_{s, s} \\
    &\quad- (1+\frac{1}{\rho^k})(\rho^{k-1}+1)\angs{\sigma_3}{\sigma_1 + (1+\rho^{k-1})\sigma_3} - \rho^k(\rho^{k-1}+1)\norm{\sigma_3}^2\\
    &\quad+ (\rho-\frac{1}{\rho^k})(\rho^{k-1}+1)\angs{\sigma_3}{\sigma_4} + \rho^k\angs{\sigma_3}{\sigma_2 - \sigma_1} - \rho\norm{\sigma_4}^2 - \rho^k\angs{\sigma_3}{\sigma_4} \\
    &\quad+ \db{\frac{1}{2}\Tr(\widehat{V} S^{(k+1)} \widehat{V}^T)}_{s, s} \,,
\end{align*}
where $V = \begin{bmatrix}
        x_0 - x_* \ | \ s_1 \ | \ \dots \ | \ s_n \ | \ s_*
    \end{bmatrix}$ and $\widetilde{V} = \begin{bmatrix}
        x_{n+1} - x_* \ | \ s_{n+2} \ | \ \dots \ | \ s_{2n+1} \ | \ s_*
    \end{bmatrix}$. Collecting the terms that involve vectors other than those from $\sigma$, by the recursive construction of $S^{(k+1)}$ (Definition \ref{def:laplacian}) we have
\begin{align*}
    &\rho^2\angs{\sigma_1 + (\rho^{k-1}+1)\sigma_3}{s_{n+2}} -\db{\frac{1}{2}\Tr(V S V^T)}_{s, s} - \db{\frac{\rho^2}{2}\Tr(\widetilde{V} S \widetilde{V}^T)}_{s, s} + \db{\frac{1}{2}\Tr(\widehat{V} S^{(k+1)} \widehat{V}^T)}_{s, s} \\
    &= \cancel{\rho^2\angs{\sigma_1 + (\rho^{k-1}+1)\sigma_3}{s_{n+2}}} + \frac{1}{2}\norm{\sigma_2 - \sigma_1}^2 + \frac{\rho^2}{2}\norm{\sigma_5 - \sigma_4}^2 - \cancel{\rho^2\angs{\sigma_1 + (\rho^{k-1}+1)\sigma_3}{s_{n+2}}} \\
    &\quad+ \frac{1}{2}(\rho^{k-1}+1)(\rho^{k+1}+1)\norm{\sigma_3}^2 - \rho^k\angs{\sigma_3}{\sigma_2 - \sigma_1} - \rho\angs{\sigma_3}{\sigma_4} \\
    &\quad- \frac{1}{2}\norm{\frac{\rho^{k-1}+1}{\rho^k}\sigma_3-(\rho - \frac{1}{\rho^k})\sigma_4 + \rho \sigma_5}^2\,,
\end{align*}
where the last term is from $c^{(k+1)} - \pi^{(k+1)} = [\underbrace{0, \dots, 0}_{n}, \frac{\rho^{k-1}+1}{\rho^k}, \rho c^{(k)} - (\rho-\frac{1}{\rho^k})\pi^{(k)}]$. Thus, it suffices to show that
\begin{align*}
    &\angs{\sigma_1}{\sigma_2 - \sigma_1} + \angs{\sigma_1 + (\rho^{k-1}+1)\sigma_3 + \sigma_4}{(\rho-1-\frac{1}{\rho^k})\sigma_4 + (\rho+1)\sigma_5} -\frac{1}{2}\norm{\sigma_2}^2 \\
    &\quad- \rho^2 \angs{\sigma_1 + (\rho^{k-1}+1)\sigma_3}{\sigma_5} -\frac{\rho^2}{2}\norm{\sigma_5}^2 -(1+\frac{1}{\rho^k})(\rho^{k-1}+1)\angs{\sigma_3}{\sigma_1 + (\rho^{k-1}+1)\sigma_3} \\
    &\quad -\rho^k(\rho^{k-1}+1)\norm{\sigma_3}^2+ (\rho-\frac{1}{\rho^k})(\rho^{k-1}+1)\angs{\sigma_3}{\sigma_4} + \rho^k\angs{\sigma_3}{\sigma_2 - \sigma_1} - \rho\norm{\sigma_4}^2 \\
    &\quad- \rho^k\angs{\sigma_3}{\sigma_4} + \frac{1}{2}\norm{\sigma_2 - \sigma_1}^2 + \frac{\rho^2}{2}\norm{\sigma_5 - \sigma_4}^2 + \frac{1}{2}(\rho^{k-1}+1)(\rho^{k+1}+1)\norm{\sigma_3}^2 \\
    &\quad- \rho^k\angs{\sigma_3}{\sigma_2 - \sigma_1} - \rho\angs{\sigma_3}{\sigma_4} - \frac{1}{2}\norm{\frac{\rho^{k-1}+1}{\rho^k}\sigma_3-(\rho - \frac{1}{\rho^k})\sigma_4 + \rho \sigma_5}^2
\end{align*}
is equal to $-\frac{1}{2}\norm{\sigma w_s}^2$. 
This can be checked by hand; for brevity, we provide a simple Mathematica script that rigorously verifies these identities in the supplementary file~\citep{MathematicaURL}.
\end{proof}
Finally, we verify the terms involving $x_0 - x_*$ and $s_*$.

\begin{proposition}[Quadratic form involving $x_0 - x_*$]\label{prop:init} Consider the setting of Theorem \ref{thm:wts} and let $\gamma, w_g, \sigma, w_s$ be as in Definition \ref{def:qfvectors}. Then
\begin{align*}
    \db{R + T^f}_{i, g} &= \angs{x_0 - x_*}{\gamma w_g}\,,\\
    \db{R + T^h}_{i, s} &= \angs{x_0 - x_*}{s_1} + \angs{x_0 - x_*}{\sigma w_s}\,,\\
    \db{R + T^f + T^h}_{i, i} &= 0\,.
\end{align*}
\end{proposition}
\begin{proof} The last equality $\db{R + T^f + T^h}_{i, i} = 0$ is clear from inspection. For other cases,
\begin{align*}
    \db{R + T^f}_{i, g} &= \sum_{i=0}^{n-1}(-\alpha_i + \alpha_i + \alpha_i)\angs{x_0 - x_*}{g_i} + \sum_{i=0}^{n-1}(-\rho^k + \rho^k + (\rho^{k-1}+1))\angs{x_0 - x_*}{g_n} \\
    &\quad+ \sum_{i=n+1}^{2n}(-\rho^2 \alpha_i + \rho^2 \alpha_i + \alpha_i)\angs{x_0 - x_*}{g_i} \\
    &\quad+ (-\rho^{k+2}+\rho^{k+2} + \rho^{k+1})\angs{x_0 - x_*}{g_{2n+1}} \\
    &= \angs{x_0 - x_*}{\gamma w_g}\,, \\
    \db{R + T^h}_{i, s} &= \angs{x_0 - x_*}{s_1} + \sum_{j=1}^n (-c_j + c_j + \alpha_{j-1})\angs{x_0 - x_*}{s_j} \\
    &\quad + (1+\frac{1}{\rho^k})(\rho^{k-1}+1)\angs{x_0 - x_*}{s_{n+1}} - \rho^2 \angs{x_0 - x_*}{s_{n+2}} + \rho^2 \angs{x_0 - x_*}{s_{n+2}} \\
    &\quad+ \sum_{j=n+2}^{2n+1}(- \rho^2 c_{j-n-1} + \rho^2 c_{j-n-1} + \rho(c_{j-n-1} - \alpha_{j-1}) + (1+\frac{1}{\rho^k})\alpha_{j-1})\angs{x_0 - x_*}{s_j} \\
    &= \angs{x_0 - x_*}{s_1} + \angs{x_0 - x_*}{\sigma w_s}\,.
\end{align*}
\end{proof}

\begin{proposition}[Quadratic form involving $s_*$]\label{prop:terminal}  Consider the setting of Theorem \ref{thm:wts} and let $\gamma, w_g,$ be as in Definition \ref{def:qfvectors}. Then
\begin{align*}
    \db{R + T^f}_{o, g} &= -\angs{s_*}{\gamma w_g}\,,\\
    \db{R + T^f}_{o, s} &= 0\,, \\
    \db{R + T^f}_{o, i} &= 0\,, \\
    \db{R + T^f}_{o, o} &= -\frac{1}{2}(2\rho^{k+1}-1)\norm{s_*}^2\,.
\end{align*}
\end{proposition}
\begin{proof} The second and third equalities $\db{R + T^f}_{o, s} = \db{R + T^f}_{o, i} = 0$ are clear from inspection. For other cases,
\begin{align*}
    &\db{R + T^f}_{o, g} \\
    &= \sum_{i=0}^{n-1}(\alpha_i - \alpha_i - \alpha_i)\angs{s_*}{g_i} + (\rho^k - \rho^k - (\rho^{k-1}+1))\angs{s_*}{g_n} \\
    &\quad + \sum_{i=n+1}^{2n}(\rho^2 \alpha_i - \rho^2 \alpha_i - \alpha_i)\angs{s_*}{g_i} + (\rho^{k+2} - \rho^{k+2} - \rho^{k+1})\angs{s_*}{g_{2n+1}} \\
    &= -\angs{\gamma w_g}{s_*}\,,\\
    &\db{R + T^f}_{o, o} \\
    &= \left(\sum_{i=0}^{n-1} \frac{\alpha_i}{2} + \frac{\rho^k}{2} - \frac{1}{2} + \sum_{i=n+1}^{2n} \frac{\rho^2 \alpha_i}{2} + \frac{\rho^{k+2}}{2} - \frac{\rho^2}{2} - \sum_{i=0}^{n-1} \frac{\alpha_i}{2} - \frac{\rho^{k-1}+1}{2} - \sum_{i=0}^{n-1} \frac{\alpha_i}{2} - \frac{\rho^{k+1}}{2}\right)\norm{s_*}^2 \\
    &= -\frac{1}{2}(2\rho^{k+1}-1)\norm{s_*}^2\,.
\end{align*}    
\end{proof}

Combining these results proves Theorem \ref{thm:wts}.
\begin{proof}[Proof of Theorem \ref{thm:wts}] 
The propositions in Appendices \ref{appendix:linear} and \ref{appendix:quadproof} show that the left and right hand side of~\eqref{eq:wts} match.
\end{proof}

\section{Expanded expression for the decomposition}

For the reader's convenience, here we provide a fully expanded expression for the terms in \eqref{eq:TCdecomp}. This follows by expanding the definition of co-coercivities (Definition \ref{def:coco}) and using the induction hypothesis \eqref{eq:maincert}.

\begin{proposition}[Expanded expressions for \eqref{eq:TCdecomp}]\label{prop:decompdetail} Let $R, T^f, T^h, C^f, C^h$ be as in \eqref{eq:TCdecomp}, and $u, V, S$ be as in Theorem \ref{thm:maincert}. Also, assume that $\eqref{eq:maincert}$ is true for $k$. Then with
\begin{align*}
    \widetilde{u} &:= x_{n+1} - x_* - \sum_{i=n+1}^{2n} \alpha_i g_i - \rho^k g_{2n+1} - \sum_{j=n+2}^{2n+1} c_{j-n-1} s_j - s_*\,, \\
    \widetilde{V} &:= \begin{bmatrix}
        x_{n+1} - x_* \ | \ s_{n+2} \ | \ \dots \ | \ s_{2n+1} \ | \ s_*
    \end{bmatrix}\,,
\end{align*}
the explicit expressions for the terms in the decomposition are given as follows:

\begingroup
\allowdisplaybreaks
\begin{align*}
    R &= \underbrace{\sum_{i=0}^{n-1}\alpha_i (f_i - f_n + \angs{g_i}{x_* - x_i} + \frac{1}{2}\norm{g_* - g_i}^2)}_{A_1} + \underbrace{\rho^k \angs{g_n}{x_* - x_n}}_{B_1} + \frac{\rho^k}{2}\norm{g_* - g_n}^2 \\
    &\quad+ \underbrace{h_1 - h_n + \angs{s_1}{x_* - x_1}}_{B_2} + \underbrace{\sum_{j=1}^n c_j(h_j - h_n + \angs{s_j}{x_* - x_j})}_{A_1} + \frac{1}{2}(1+\frac{1}{\sqrt{2}})\norm{x_0 - x_*}^2 \\
    &\quad- \frac{1}{2}(\norm{u}^2 + \Tr(VSV^T)) + \underbrace{\sum_{i=n+1}^{2n}\rho^2\alpha_i(f_i - f_{2n+1} + \angs{g_i}{x_* - x_i} + \frac{1}{2}\norm{g_* - g_i}^2)}_{A_2}  \\
    &\quad+ \underbrace{\rho^{k+2}\angs{g_{2n+1}}{x_* - x_{2n+1}}}_{B_3} + \frac{\rho^{k+2}}{2}\norm{g_* - g_{2n+1}}^2 \\
    &\quad+ \underbrace{\rho^2(h_{n+2}-h_{2n+1} + \angs{s_{n+2}}{x_* - x_{n+2}})}_{B_2} + \underbrace{\sum_{j=n+2}^{2n+1}\rho^2 c_{j-n-1}(h_j - h_{2n+1} + \angs{s_j}{x_* - x_j})}_{A_3} \\
    &\quad+ \frac{\rho^2}{2}(1+\frac{1}{\sqrt{2}})\norm{x_{n+1} - x_*}^2 - \frac{\rho^2}{2}(\norm{\widetilde{u}}^2 + \Tr(\widetilde{V} S \widetilde{V}^T)) \,, \\
    T^f &= \underbrace{\sum_{i=0}^{n-1}\alpha_i (f_* - f_i - \angs{g_i}{x_* - x_i} - \frac{1}{2}\norm{g_* - g_i}^2)}_{A_1} \\
    &\quad+ (\rho^{k-1}+1)(f_* - f_n - \angs{g_n}{x_* - x_n} - \frac{1}{2}\norm{g_* - g_n}^2) \\
    &\quad+ \underbrace{\sum_{i=n+1}^{2n}\alpha_i (f_* - f_i - \angs{g_i}{x_* - x_i} - \frac{1}{2}\norm{g_* - g_i}^2)}_{A_2} \\
    &\quad+ \underbrace{\rho^{k+1}(f_* - f_{2n+1} - \angs{g_{2n+1}}{x_* - x_{2n+1}} - \frac{1}{2}\norm{g_* - g_{2n+1}}^2)}_{B_3}\,, \\
    T^h &= \underbrace{h_* - h_1 - \angs{s_1}{x_* - x_1}}_{B_1} + \sum_{j=1}^n \alpha_{j-1}(h_* - h_j - \angs{s_j}{x_* - x_j}) \\
    &\quad+ (1+\frac{1}{\rho^k})(\rho^{k-1} + 1)(h_* - h_{n+1} - \angs{s_{n+1}}{x_* - x_{n+1}})\\
    &\quad+ \underbrace{\sum_{j=n+2}^{2n+1}(\rho(c_{j-n-1} - \alpha_{j-1}) + (1 + \frac{1}{\rho^k})\alpha_{j-1})(h_* - h_j - \angs{s_j}{x_* - x_j})}_{A_3}\,, \\
    C^f &= \underbrace{\sum_{i=n+1}^{2n}\rho \alpha_i(f_n - f_i - \angs{g_i}{x_n - x_i} - \frac{1}{2}\norm{g_n - g_i}^2)}_{A_2} \\
    &\quad+ \rho(f_n - f_{2n+1} - \angs{g_{2n+1}}{x_n - x_{2n+1}} - \frac{1}{2}\norm{g_n - g_{2n+1}}^2) \\
    &\quad+ \sum_{i=n+1}^{2n}\rho\alpha_i(f_{2n+1} - f_i - \angs{g_i}{x_{2n+1} - x_i} - \frac{1}{2}\norm{g_{2n+1} - g_i}^2) \\
    &\quad+ \underbrace{\rho^k (f_{2n+1} - f_n - \angs{g_n}{x_{2n+1} - x_n} - \frac{1}{2}\norm{g_n - g_{2n+1}}^2)}_{B_1}\,,
\end{align*}
\endgroup
and
\begin{align*}
    C^h &= \rho^k(h_n - h_{n+1} - \angs{s_{n+1}}{x_n - x_{n+1}}) + \underbrace{\rho^2(h_{n+1} - h_{n+2} - \angs{s_{n+2}}{x_{n+1} - x_{n+2}})}_{B_2} \\
    &\quad+ (\rho - \frac{1}{\rho^k})(\rho^{k-1} + 1)(h_{2n+1}-h_{n+1} - \angs{s_{n+1}}{x_{2n+1} - x_{n+1}}) \\
    &\quad+ \underbrace{\sum_{j=1}^n(c_j - \alpha_{j-1})(h_n - h_j - \angs{s_j}{x_n - x_j})}_{A_1} \\
    &\quad+ \sum_{j=1}^n \frac{\rho^k}{\rho^{k-1}+1}(c_j-\alpha_{j-1})(h_{n+1}-h_n - \angs{s_j}{x_{n+1}-x_n})\\
    &\quad+ \sum_{j=n+2}^{2n+1}\rho \alpha_{j-1}(h_{n+1} - h_{2n+1} - \angs{s_j}{x_{n+1} - x_{2n+1}}) \\
    &\quad+ \sum_{j=n+2}^{2n+1}\frac{\rho^k}{\rho^{k-1}+1}\alpha_{j-1}(h_n - h_{n+1}-\angs{s_j}{x_n-x_{n+1}})\\
    &\quad \underbrace{\sum_{j=n+2}^{2n+1}((\rho+1)c_{j-n-1} + (\rho-1-\frac{1}{\rho^k})\alpha_{j-1})(h_{2n+1} - h_j - \angs{s_j}{x_{2n+1} - x_j})}_{A_3}\,.
\end{align*}
\end{proposition}
\begin{proof}
    The equalities for $T^f, T^h, C^f$ are obtained directly from expanding the co-coercivities. The expression for $C^h$ can be obtained after combining terms sequentially; in particular,
    \begin{align*}
        &\sum_{j=1}^n \bmu_{n, j}^{(k+1), \lr}Q_{nj}^h + \sum_{j=1}^n \bmu_{n+1, j}^{(k+1), \lr}Q_{(n+1)j}^h \\
        &=\sum_{j=1}^n (1-\frac{\rho^k}{\rho^{k-1}+1})(c_j - \alpha_{j-1})(h_n - h_j - \angs{s_j}{x_n - x_j}) \\
        &\quad+ \sum_{j=1}^n \frac{\rho^k}{\rho^{k-1}+1}(c_j - \alpha_{j-1})(h_{n+1}-h_j - \angs{s_j}{x_{n+1}-x_j})\\
        &= \sum_{j=1}^n (c_j - \alpha_{j-1})(h_n - h_j - \angs{s_j}{x_n - x_j}) \\
        &\quad+ \sum_{j=1}^n \frac{\rho^k}{\rho^{k-1}+1}(c_j - \alpha_{j-1})(h_{n+1}-h_n - \angs{s_j}{x_{n+1} - x_n})\,, \\
        &\sum_{j=n+2}^{2n+1} \bmu_{n, j}^{(k+1), \lr}Q_{nj}^h + \sum_{j=n+2}^{2n+1} \bmu_{n+1, j}^{(k+1), \lr}Q_{(n+1)j}^h \\
        &=\sum_{j=n+2}^{2n+1} \frac{\rho^k}{\rho^{k-1} + 1}\alpha_{j-1}(h_n - h_j - \angs{s_j}{x_n - x_j}) \\
        &\quad+ \sum_{j=n+2}^{2n+1} \frac{\rho}{\rho^{k-1} + 1}\alpha_{j-1}(h_{n+1} - h_j - \angs{s_j}{x_{n+1} - x_j}) \\
        &= \underbrace{\sum_{j=n+2}^{2n+1} \rho \alpha_{j-1}(h_{n+1} - h_j - \angs{s_j}{x_{n+1} - x_j})}_{=:\Delta} \\
        &\quad+ \sum_{j=n+2}^{2n+1} \frac{\rho^k}{\rho^{k-1} + 1} \alpha_{j-1}(h_n - h_{n+1} - \angs{s_j}{x_n - x_{n+1}})\,,
    \end{align*}
    and
    \begin{align*}
        &\Delta + \sum_{j = n+2}^{2n+1}\bmu_{2n+1, j}^{(k+1), \lr}Q_{(2n+1)j}^h \\
        &= \Delta + \sum_{j=n+2}^{2n+1}((\rho+1)c_{j-n-1} - (1+\frac{1}{\rho^k})\alpha_{j-1})(h_{2n+1} - h_j - \angs{s_j}{x_{2n+1}-x_j}) \\
        &= \sum_{j=n+2}^{2n+1} \rho \alpha_{j-1} (h_{n+1} - h_{2n+1} - \angs{s_j}{x_{n+1} - x_{2n+1}})\\
        &\quad+ \sum_{j=n+2}^{2n+1}((\rho+1)c_{j-n-1} + (\rho - 1-\frac{1}{\rho^k})\alpha_{j-1})(h_{2n+1} - h_j - \angs{s_j}{x_{2n+1}-x_j})\,.
    \end{align*}
    The expression for $R$ can be obtained directly from the definition of $\ulambda^{(k)}$ and \eqref{eq:maincert}, which is valid for the iterates $x_0, x_1, \dots, x_n, x_*$ of proximal GD with the silver stepsize schedule $\pi^{(k)}$ and for the iterates $x_{n+1}, x_{n+2}, \dots, x_{2n+1}, x_*$ of the same algorithm with initial point $x_{n+1}$, by the recursive construction of the silver stepsize schedule $\pi^{(k+1)} = [\pi^{(k)}, \rho^{k-1}+1, \pi^{(k)}]$ (Lemma \ref{lem:silver}).
\end{proof}

\end{document}